\documentclass[12pt]{article}

\usepackage{amsmath, amssymb, amsthm}
\usepackage{graphicx}
\graphicspath{{figures/}{./}}
\usepackage{booktabs}
\usepackage{hyperref}
\usepackage{geometry}
\usepackage{caption}
\usepackage{subcaption}
\usepackage{bm}
\usepackage{float}
\usepackage{doi}
\usepackage{placeins}
\allowdisplaybreaks
\hypersetup{
    pdfencoding=auto,
    psdextra=true,
    unicode=true,
    hidelinks
}

\newtheorem{theorem}{Theorem}
\newtheorem{proposition}{Proposition}
\newtheorem{corollary}{Corollary}
\newtheorem{definition}{Definition}
\newtheorem{lemma}{Lemma}
\newtheorem{problem}{Problem}

\newtheorem{remark}{Remark}

\title{Fixed Points, Stability, Basin Geometry, and Global Convergence of the
\(3\times3\) Correlation Map}

\author{Ishrak Alhajj Hassan\\
\small Department of Mathematics, Faculty of Science, University of Ostrava,\\
\small Ostrava, Czech Republic\\
\small Email:
\href{mailto:ishrak.alhajj.s01@osu.cz}
{\texttt{ishrak.alhajj.s01@osu.cz}}\\
\small ORCID:
\href{https://orcid.org/0009-0009-5411-9799}
{0009-0009-5411-9799}}

\date{}

\begin{document}

\maketitle

\begin{abstract}
For more than half a century, iterated Pearson correlation has underpinned
methods for network blockmodeling, clustering, seriation, and information
visualization. Despite its continued use and the longstanding assumption of
convergence, the global convergence problem remained open even in dimension
three. We resolve this problem completely for the \(3\times3\) correlation
map and obtain complete fixed-point and relative Lyapunov stability
classifications.
We first establish an exact row-wise Gram factorization of the correlation
map and the identity
\[
\operatorname{rank}C(A)=\operatorname{rank}(AH_n),
\qquad
H_n=I_n-\frac1n\mathbf1\mathbf1^T,
\]
where \(I_n\) is the \(n\times n\) identity matrix and \(H_n\) is the
centering matrix. In arbitrary dimension, this yields the precise rank-reduction mechanism,
forces every fixed point to be singular, and establishes a forward-invariance
theorem for the Pearson-nondegenerate elliptope. Here
\(\mathcal C_n\) denotes the elliptope of \(n\times n\) correlation
matrices. More precisely, the all-ones matrix is the unique
Pearson-degenerate correlation matrix and
\[
C\bigl(\mathcal C_n^{\mathrm{nd}}\bigr)
\subseteq
\mathcal C_n^{\mathrm{nd}},
\qquad
\mathcal C_n^{\mathrm{nd}}
=
\mathcal C_n\setminus\{\mathbf1\mathbf1^T\}.
\]
Thus every trajectory starting in the natural nondegenerate state space is
defined for all forward iterates. We also characterize all nondegenerate
sign-valued fixed points in arbitrary dimension and recover their exact
number, \(2^{n-1}-1\).

For \(n=3\), we derive an explicit coordinate representation and prove that
the natural nondegenerate domain contains exactly seven fixed points: three
rank-one patterned points, three rank-two mixed points, and one rank-two
equicorrelation point. Exact Jacobian calculations together with admissible
invariant boundary curves give the complete relative Lyapunov stability
classification: precisely the patterned points are locally asymptotically
stable.

The principal new global result is that every admissible \(3\times3\) orbit
converges to one of these seven fixed points. After one iteration, the exact
rank identity confines the dynamics to rank at most two. For a rank-two
iterate \(P=XX^T\), we use the one-dimensional kernel
\(\ker P=\ker X^T\) as a dynamical coordinate. If
\(q\ne0\) spans this kernel and
\(D=\operatorname{diag}(d_1,d_2,d_3)\) contains the centered-row norms, then
the kernel of the next rank-two iterate is spanned by \(Dq\). An exact
identity,
\[
d_i^2-d_j^2
=
\frac{2\det(X^TH_3X)}{q_1+q_2+q_3}(q_i-q_j),
\]
whenever \(q_1+q_2+q_3\ne0\), yields monotone projective ratios and exhausts
all sign and boundary cases. This proves global convergence and excludes
nontrivial periodic or recurrent limit sets in dimension three.

Finally, the exceptional initial conditions converging to the four unstable
fixed points form a Lebesgue-null set. Hence almost every admissible initial
condition converges to one of the three patterned fixed points. Their basins
are relatively open, permutation-equivalent, and have equal Lebesgue measure;
consequently each patterned basin has exactly one third of the elliptope's
Lebesgue measure. Reproducible basin computations are retained as an
independent numerical illustration of the exact global theory.
\end{abstract}



\section{Introduction}\label{sec:introduction}

For a real \(n\times n\) matrix \(A\) with rows
\(\mathbf a_1,\ldots,\mathbf a_n\in\mathbb R^n\), the correlation map
\(C(A)\) is defined as the matrix whose \((i,j)\)-entry is the Pearson
correlation coefficient between rows \(\mathbf a_i\) and \(\mathbf a_j\).
Whenever all rows of \(A\) are nonconstant, \(C(A)\) is symmetric and
positive semidefinite with unit diagonal and therefore belongs to the
elliptope \(\mathcal C_n\). Repeated application,
\[
A_{k+1}=C(A_k),
\]
defines a nonlinear discrete dynamical system once the natural
nondegenerate state space is identified.

The iteration originated in multivariate data analysis and underlies the
classical CONCOR clustering procedure of Breiger, Boorman, and Arabie
\cite{breiger1975algorithm} and the generalized association-plot framework of
Chen~\cite{chen2002generalized}. These methods exploit the strong empirical
tendency of repeated correlation to produce structured matrices with entries
in \(\{-1,1\}\). Chen established a dimension-independent rank-reduction
property and displayed stationary configurations in low dimensions. Kruskal
proved a quantitative local convergence criterion near a prescribed two-class
sign pattern~\cite{kruskalCONCOR}. Extensive numerical evidence has also
supported rapid convergence across a broad range of dimensions
\cite{alhajj2025empirical}. Complementary finite-step probabilistic bounds for
contraction ratios were established in \cite{alhajjhassan2026finite}.

The main analytical difficulty is global. The Pearson map is nonlinear and is
undefined when a row becomes constant. Fixed points lie on singular boundary
strata of the elliptope, where ambient linearization alone need not determine
relative stability. Even a complete fixed-point classification does not by
itself exclude periodic orbits, recurrent boundary dynamics, or other
nontrivial omega-limit sets.

The first part of this paper develops the structural framework needed to
separate these issues. Building on Chen's rank-reduction mechanism
\cite[Lemma~3.1]{chen2002generalized}, we establish the exact row-wise Gram
factorization
\[
C(A)=\Delta^{-1}AH_nA^T\Delta^{-1},
\]
where \(I_n\) is the \(n\times n\) identity matrix,
\[
H_n=I_n-\frac1n\mathbf1\mathbf1^T
\]
is the centering matrix, and \(\Delta\) is the diagonal matrix whose entries
are the Euclidean norms of the centered rows of \(A\). Hence
\[
\operatorname{rank}C(A)=\operatorname{rank}(AH_n).
\]
We then identify the all-ones matrix as the unique Pearson-degenerate point
of the elliptope and prove that the Pearson-nondegenerate elliptope is forward
invariant. Writing
\[
\mathcal C_n^{\mathrm{nd}}
=
\mathcal C_n\cap\mathcal D_n
\]
for the correlation matrices with nonconstant rows, the restriction of
\(C\) to \(\mathcal C_n^{\mathrm{nd}}\) is therefore a genuine self-map,
and finite-time loss of nondegeneracy is excluded in every dimension.

Before specializing to dimension three, we also characterize all
nondegenerate sign-valued fixed points. Every sign-valued correlation matrix
has the form \(ss^T\) with \(s\in\{-1,1\}^n\); excluding the constant sign
class yields exactly
\[
2^{n-1}-1
\]
nondegenerate sign-valued fixed matrices.

For \(n=3\), Chen~\cite[Figure~5]{chen2002generalized} displayed the three
structural types of stationary configurations. Theorem~\ref{thm:complete-classification}
gives a self-contained exhaustive analytical classification on the natural
nondegenerate domain. There are exactly seven fixed points: three rank-one
patterned points, three rank-two mixed points, and one rank-two
equicorrelation point, forming three permutation orbits under \(S_3\),
the symmetric group on three symbols.

Theorem~\ref{thm:stability} then gives the complete relative Lyapunov
stability classification. The patterned points are locally asymptotically
stable. The mixed and equicorrelation points are Lyapunov unstable relative
to the elliptope. Their instability is proved not merely from ambient
expanding eigenvalues, but by invariant boundary curves that realize the
multipliers
\[
\frac{4\sqrt3}{3}
\qquad\text{and}\qquad
\frac32
\]
inside the admissible state space.

The central result of the paper is the global convergence theorem,
Theorem~\ref{thm:global-convergence}. The exact rank identity implies that
after one iteration every \(3\times3\) orbit has rank at most two. The
remaining rank-two dynamics is represented by the projective kernel direction
of the current correlation matrix. If \(P=XX^T\) has rank two and
\(q\in\ker P\setminus\{0\}\), then the next kernel is obtained by positive
diagonal rescaling. A second exact identity orders the diagonal rescaling
factors in exactly the same way as the coordinates of \(q\). This produces
monotone projective ratios. A complete analysis of the two possible sign
chambers and their boundary cases proves that every admissible orbit converges
to one of the seven fixed points.

The same analysis confines convergence to the four non-patterned fixed
points to an exceptional set. Combined with real analyticity of the
coordinate map in the positive-definite interior, it implies that the union
of the three patterned basins has full Lebesgue measure. The basins are relatively open and permutation-equivalent, so each
has exactly one third of the elliptope's Lebesgue measure. Thus the nearly
equal frequencies seen in the numerical basin experiment are the finite-sample
counterpart of an exact measure-theoretic statement rather than evidence for
an unresolved conjecture.

The numerical experiment is retained for reproducibility and geometric
illustration. It treats all seven fixed points uniformly and independently
confirms that each of the \(49{,}164\) admissible sampled trajectories is
classified as converging to one of the three patterned fixed points. The exact
global theorem, however, is entirely analytical and does not depend on the
Monte Carlo computation.

Taken together, the paper establishes forward well-posedness on the natural
state space in arbitrary dimension, a complete fixed-point and relative
stability theory in dimension three, and a global convergence theorem for
every admissible \(3\times3\) orbit. The higher-dimensional convergence
problem remains open.

\section{Preliminaries}\label{sec:preliminaries}

Let \(M_n(\mathbb R)\) denote the space of real \(n\times n\)
matrices. For each positive integer \(m\), let \(I_m\) denote the
\(m\times m\) identity matrix and let
\(\mathbf1_m\in\mathbb R^m\) denote the all-ones vector. We omit the
subscript and write \(\mathbf1\) whenever its dimension is determined
by the surrounding formula. We write
\[
\mathbf1_m^\perp
=
\{x\in\mathbb R^m:x^T\mathbf1_m=0\}
\]
for the orthogonal complement of \(\mathbf1_m\).
We use \(\langle\cdot,\cdot\rangle\) and \(\|\cdot\|\) for the Euclidean
inner product and norm, respectively. For a matrix \(M\), the notation
\(M_{i,:}\) denotes its \(i\)-th row. For a symmetric matrix \(A\),
the notations \(A\succeq0\) and \(A\succ0\) mean that \(A\) is
positive semidefinite and positive definite, respectively. We use
\(\ker\), \(\operatorname{rank}\), \(\operatorname{tr}\), \(\det\),
\(\operatorname{diag}\), and \(\operatorname{span}\) with their standard
linear-algebraic meanings. For any self-map \(F\), we write
\(F^0=\operatorname{id}\) and \(F^{k+1}=F\circ F^k\) for its iterates.
\begin{definition}
For a vector \(x=(x_1,\ldots,x_m)\in\mathbb{R}^m\), define its mean by
\[
\bar{x}=\frac{1}{m}\sum_{k=1}^{m}x_k.
\]
For two nonconstant vectors \(x,y\in\mathbb{R}^m\), the Pearson
correlation coefficient is
\[
\operatorname{corr}(x,y)
=
\frac{
\sum_{k=1}^{m}(x_k-\bar{x})(y_k-\bar{y})
}{
\sqrt{
\sum_{k=1}^{m}(x_k-\bar{x})^2
\sum_{k=1}^{m}(y_k-\bar{y})^2
}
}.
\]
\end{definition}

\begin{definition}
The set of real \(n\times n\) correlation matrices is
\[
\mathcal C_n
=
\left\{
A\in M_n(\mathbb{R}):
A=A^T,\;
A\succeq0,\;
A_{ii}=1\text{ for }i=1,\ldots,n
\right\}.
\]
It is a compact convex subset of the cone of positive semidefinite
matrices and plays a central role in matrix analysis and numerical
optimization; see \cite{horn2013matrix,higham2002}.
\end{definition}

\begin{definition}
Let
\[
\mathcal D_n
=
\left\{
A\in M_n(\mathbb{R}):
\text{all rows of }A\text{ are nonconstant}
\right\}.
\]
For \(A\in\mathcal D_n\), write its rows as
\(r_1,\ldots,r_n\). The correlation map is the operator
\[
C:\mathcal D_n\longrightarrow M_n(\mathbb R)
\]
defined by
\[
[C(A)]_{ij}
=
\operatorname{corr}(r_i,r_j),
\qquad
1\le i,j\le n.
\]

A matrix \(A\in\mathcal D_n\) is called a \emph{fixed point} of \(C\)
if
\[
C(A)=A.
\]
\end{definition}

\begin{proposition}\label{prop:corrmatrix}
For every matrix \(A\in\mathcal D_n\),
\[
C(A)\in\mathcal C_n.
\]
\end{proposition}
\begin{proof}
Let \(r_i\) be the \(i\)-th row of \(A\), and set
\begin{equation}\label{eq:normalized-centered-row}
u_i
=
\frac{r_i-\bar r_i\mathbf1}
{\|r_i-\bar r_i\mathbf1\|},
\qquad
i=1,\ldots,n.
\end{equation}
Since \(A\in\mathcal D_n\), the denominator is nonzero for every
\(i\). Then
\[
[C(A)]_{ij}
=
\langle u_i,u_j\rangle.
\]
Hence
\begin{equation}\label{eq:correlation-gram-form}
C(A)=UU^T,
\end{equation}
where the rows of \(U\) are \(u_1,\ldots,u_n\). Therefore \(C(A)\)
is symmetric, positive semidefinite, and has diagonal entries equal to
one.
\end{proof}

The dimension-independent rank-reduction property of the correlation
map was established by Chen~\cite[Lemma~3.1]{chen2002generalized},
using standard rank inequalities and Sylvester's law of nullity to
control the sequence of matrix ranks. In the row-wise framework of the
present paper, we establish an exact row-wise Gram factorization of the correlation
map. After centering by \(H_n\) and normalizing by the invertible
diagonal matrix \(\Delta^{-1}\), this factorization yields the exact
rank identity and connects rank reduction directly to the geometry of
the fixed-point set.

\begin{proposition}[Exact row-wise factorization and rank reduction]
\label{prop:rank-bound}
Let \(A\in\mathcal D_n\), and let
\[
H_n
=
I_n-\frac1n\mathbf1\mathbf1^T
\]
be the centering matrix, equivalently the orthogonal projector onto
\(\mathbf1_n^\perp\). Define
\[
\Delta
=
\operatorname{diag}
\bigl(
\|(AH_n)_{1,:}\|,\ldots,\|(AH_n)_{n,:}\|
\bigr).
\]
Then \(\Delta\) is invertible and
\begin{equation}\label{eq:correlation-centered-factorization}
C(A)
=
\Delta^{-1}AH_nA^T\Delta^{-1}.
\end{equation}
In particular,
\begin{equation}\label{eq:rank-correlation-centered}
\operatorname{rank}C(A)
=
\operatorname{rank}(AH_n).
\end{equation}
Consequently,
\begin{equation}\label{eq:rank-reduction-bounds}
\operatorname{rank}A-1
\le
\operatorname{rank}C(A)
\le
\min\{\operatorname{rank}A,n-1\}.
\end{equation}
Every fixed point of \(C\) is therefore singular and lies on the
relative boundary of \(\mathcal C_n\).
\end{proposition}

\begin{proof}
Set
\[
B=AH_n.
\]
The \(i\)-th row of \(B\) is the centered \(i\)-th row of \(A\).
Since \(A\in\mathcal D_n\), none of these rows vanishes, and therefore
the matrix \(\Delta\) defined in the statement is invertible. The
matrix of normalized centered rows is
\[
U=\Delta^{-1}B=\Delta^{-1}AH_n.
\]
Since \(C(A)=UU^T\), using \(H_n^T=H_n\) and \(H_n^2=H_n\) gives
equation~\eqref{eq:correlation-centered-factorization}.

For every real matrix \(U\),
\[
\operatorname{rank}(UU^T)=\operatorname{rank}U.
\]
Since multiplication by the invertible matrix \(\Delta^{-1}\) does not
change rank, equation~\eqref{eq:correlation-centered-factorization}
gives
\[
\operatorname{rank}C(A)
=
\operatorname{rank}U
=
\operatorname{rank}(AH_n),
\]
which proves equation~\eqref{eq:rank-correlation-centered}.

Because
\[
\operatorname{rank}H_n=n-1,
\]
the standard rank inequality gives
\[
\operatorname{rank}(AH_n)
\le
\min\{\operatorname{rank}A,n-1\}.
\]
Sylvester's rank inequality gives
\[
\operatorname{rank}(AH_n)
\ge
\operatorname{rank}A+\operatorname{rank}H_n-n
=
\operatorname{rank}A-1.
\]
Together with equation~\eqref{eq:rank-correlation-centered}, these
inequalities prove equation~\eqref{eq:rank-reduction-bounds}.

If \(A\) is fixed, then \(A=C(A)\), and hence
\[
\operatorname{rank}A
=
\operatorname{rank}C(A)
\le
n-1.
\]
Thus every fixed point is singular. The relative interior of
\(\mathcal C_n\), in the affine space of symmetric unit-diagonal
matrices, consists precisely of the positive-definite correlation
matrices. Therefore every fixed point lies on the relative boundary of
\(\mathcal C_n\).
\end{proof}

\begin{proposition}[Sign-valued fixed points]
\label{prop:sign-fixed-points}
Let \(A\in\mathcal C_n\) have all entries in \(\{-1,1\}\). Then
there exists \(s\in\{-1,1\}^n\) such that
\begin{equation}\label{eq:sign-correlation-form}
A=ss^T.
\end{equation}
Moreover, \(A\in\mathcal D_n\) if and only if \(s\) contains both
signs, and every such nondegenerate matrix is a fixed point of \(C\).
Thus the correlation map has exactly
\begin{equation}\label{eq:number-sign-fixed-points}
2^{n-1}-1
\end{equation}
nondegenerate sign-valued fixed points.
\end{proposition}

\begin{proof}
Write \(A\) as the Gram matrix of unit vectors
\(x_1,\ldots,x_n\). Since
\[
\langle x_1,x_i\rangle\in\{-1,1\},
\]
equality holds in the Cauchy--Schwarz inequality. Hence
\[
x_i=s_i x_1
\]
for some \(s_i\in\{-1,1\}\). Set \(s_1=1\). Then
\[
A_{ij}=s_i s_j,
\]
which proves equation~\eqref{eq:sign-correlation-form}.

The \(i\)-th row of \(ss^T\) is \(s_i s^T\). These rows are
nonconstant exactly when \(s\) contains both signs. In that case the
centered \(i\)-th row is
\[
s_i\bigl(s-\bar{s}\mathbf1\bigr)^T,
\]
and therefore
\[
\operatorname{corr}(s_i s,s_j s)
=
s_i s_j.
\]
Hence
\[
C(ss^T)=ss^T.
\]

Finally, \(s\) and \(-s\) determine the same matrix. There are
\(2^{n-1}\) sign-valued correlation matrices after this identification.
The equivalence class represented by the constant vector \(\mathbf1\)
gives the all-ones matrix, which lies outside \(\mathcal D_n\). This
leaves \(2^{n-1}-1\) nondegenerate sign-valued fixed points.
\end{proof}

\begin{corollary}[Rank-one correlation states]
\label{cor:rank-one-correlation-states}
Let \(P\in\mathcal C_n\) have rank one. Then
\[
P=ss^T
\]
for some \(s\in\{-1,1\}^n\). If, in addition,
\(P\in\mathcal C_n\cap\mathcal D_n\), then \(s\) contains both signs and
\(P\) is a nondegenerate sign-valued fixed point of \(C\).
\end{corollary}

\begin{proof}
Since \(P\succeq0\) and \(\operatorname{rank}P=1\), write
\(P=vv^T\) for some nonzero \(v\in\mathbb R^n\). The unit-diagonal
condition gives \(v_i^2=1\) for every \(i\), hence
\(v\in\{-1,1\}^n\). The remaining assertions follow from
Proposition~\ref{prop:sign-fixed-points}.
\end{proof}

\begin{definition}[Pearson-nondegenerate elliptope]
\label{def:nondegenerate-elliptope}
Set
\[
\mathcal C_n^{\mathrm{nd}}
=
\mathcal C_n\cap\mathcal D_n.
\]
Thus \(\mathcal C_n^{\mathrm{nd}}\) consists of the correlation matrices on
which one further Pearson step is defined.
\end{definition}

\begin{proposition}[Unique Pearson-degenerate point of the elliptope]
\label{prop:unique-degenerate-point}
For every \(n\ge2\),
\begin{equation}\label{eq:unique-degenerate-point}
\mathcal C_n^{\mathrm{nd}}
=
\mathcal C_n\setminus\{\mathbf1\mathbf1^T\}.
\end{equation}
Equivalently, the all-ones matrix is the unique correlation matrix having a
constant row.
\end{proposition}

\begin{proof}
Let \(P\in\mathcal C_n\) and suppose that one of its rows is constant. Since
\(P_{ii}=1\), that row is equal to \(\mathbf1^T\). Write
\(P=XX^T\), where the rows \(x_1^T,\ldots,x_n^T\) of \(X\) are unit
vectors. If the \(i\)-th row of \(P\) is all ones, then
\[
1=P_{ij}=\langle x_i,x_j\rangle
\]
for every \(j\). Equality in the Cauchy--Schwarz inequality gives
\(x_j=x_i\) for every \(j\), and hence
\[
P=\mathbf1\mathbf1^T.
\]
Conversely, every row of \(\mathbf1\mathbf1^T\) is constant.
\end{proof}

\begin{theorem}[Forward invariance of the Pearson-nondegenerate elliptope]
\label{thm:forward-invariance}
For every \(n\ge2\),
\begin{equation}\label{eq:forward-invariance}
C\bigl(\mathcal C_n^{\mathrm{nd}}\bigr)
\subseteq
\mathcal C_n^{\mathrm{nd}}.
\end{equation}
Consequently, every Pearson trajectory starting in
\(\mathcal C_n^{\mathrm{nd}}\) is defined for all forward iterates.
\end{theorem}

\begin{proof}
Let \(P\in\mathcal C_n^{\mathrm{nd}}\), and define
\[
Z=\Delta^{-1}PH_n,
\qquad
\Delta
=
\operatorname{diag}
\bigl(
\|(PH_n)_{1,:}\|,\ldots,\|(PH_n)_{n,:}\|
\bigr).
\]
By Proposition~\ref{prop:rank-bound},
\[
C(P)=ZZ^T\in\mathcal C_n.
\]
Suppose for contradiction that
\(C(P)\notin\mathcal C_n^{\mathrm{nd}}\). By
Proposition~\ref{prop:unique-degenerate-point},
\[
C(P)=\mathbf1\mathbf1^T.
\]
If \(z_i^T\) denotes the \(i\)-th row of \(Z\), then
\[
\langle z_i,z_j\rangle=1
\]
for all \(i,j\). Since each \(z_i\) has unit norm, all rows coincide.
Write
\[
z_1=\cdots=z_n=v,
\qquad
\|v\|=1.
\]
Because every row of \(PH_n\) is centered, \(v^T\mathbf1=0\).

Set
\[
d_i=\|(PH_n)_{i,:}\|>0,
\qquad
m_i=\frac1nP_{i,:}\mathbf1.
\]
Since \((PH_n)_{i,:}=d_i v^T\),
\[
P_{i,:}=m_i\mathbf1^T+d_i v^T.
\]
The condition \(P_{ii}=1\) gives
\[
m_i=1-d_i v_i,
\]
and hence
\[
P_{ij}=1+d_i(v_j-v_i).
\]
Using symmetry, \(P_{ij}=P_{ji}\), we obtain
\[
(d_i+d_j)(v_j-v_i)=0.
\]
Since \(d_i+d_j>0\), all coordinates of \(v\) are equal. Together with
\(v^T\mathbf1=0\), this forces \(v=0\), contradicting \(\|v\|=1\).
Therefore \(C(P)\in\mathcal C_n^{\mathrm{nd}}\).
\end{proof}

It follows from the definition that every fixed point is a correlation
matrix and is therefore symmetric with unit diagonal. In dimension
three, it can consequently be parametrized by its three off-diagonal
entries, as developed in the next section.

\section{Dynamical System on the Elliptope}
\label{sec:dynamical-system}

Every \(3\times3\) correlation matrix has the form
\begin{equation}\label{eq:correlation-parametrization}
A(a,b,c)
=
\begin{pmatrix}
1&a&b\\
a&1&c\\
b&c&1
\end{pmatrix},
\qquad
(a,b,c)\in\mathcal E_3.
\end{equation}
For \(p=(a,b,c)\in\mathcal E_3\), we use the abbreviation
\[
A(p)=A(a,b,c).
\]
Here
\begin{equation}\label{eq:elliptope-e3}
\mathcal E_3
=
\left\{
(a,b,c)\in[-1,1]^3:
1+2abc-a^2-b^2-c^2\ge0
\right\}.
\end{equation}
We identify \(\mathcal E_3\) with this subset of \(\mathbb R^3\) through
the coordinates \( (a,b,c) \), and write \(\lambda_3\) for
three-dimensional Lebesgue measure in these coordinates.

To compute \(C(A)\), regard the rows of
equation~\eqref{eq:correlation-parametrization} as vectors:
\[
r_1=(1,a,b),
\qquad
r_2=(a,1,c),
\qquad
r_3=(b,c,1).
\]
For each row \(r_i\), let
\[
\widetilde r_i
=
r_i-\frac13(r_i\cdot\mathbf1)\mathbf1,
\]
denote its centered version. Then
\[
\operatorname{corr}(r_i,r_j)
=
\frac{
\widetilde r_i\cdot\widetilde r_j
}{
\|\widetilde r_i\|\,
\|\widetilde r_j\|
},
\]
provided \(\widetilde r_i\neq0\) and \(\widetilde r_j\neq0\).

Since \(C(A(a,b,c))\) is again a correlation matrix, it has the form
\[
C(A(a,b,c))
=
\begin{pmatrix}
1&T_1(a,b,c)&T_2(a,b,c)\\
T_1(a,b,c)&1&T_3(a,b,c)\\
T_2(a,b,c)&T_3(a,b,c)&1
\end{pmatrix}.
\]
Thus the correlation map induces
\[
T:\mathcal E_3^{\mathrm{nd}}\longrightarrow\mathcal E_3^{\mathrm{nd}},
\qquad
T(a,b,c)
=
\bigl(
T_1(a,b,c),
T_2(a,b,c),
T_3(a,b,c)
\bigr),
\]
defined on
\[
\mathcal E_3^{\mathrm{nd}}
=
\left\{
(a,b,c)\in\mathcal E_3:
\widetilde r_1\neq0,\;
\widetilde r_2\neq0,\;
\widetilde r_3\neq0
\right\}.
\]

\begin{remark}\label{rem:nondegenerate-domain}
Within the elliptope, a centered row can vanish only at the all-ones
point. For example, if the first row is constant, then \(a=b=1\).
Equation~\eqref{eq:elliptope-e3} becomes
\[
-(c-1)^2\ge0,
\]
forcing \(c=1\). The other rows are analogous. Therefore
\[
\mathcal E_3^{\mathrm{nd}}
=
\mathcal E_3\setminus\{(1,1,1)\}.
\]

By Theorem~\ref{thm:forward-invariance},
\[
T\bigl(\mathcal E_3^{\mathrm{nd}}\bigr)
\subseteq
\mathcal E_3^{\mathrm{nd}},
\]
so the induced map is a genuine self-map of its natural nondegenerate domain.
In particular, every orbit starting in \(\mathcal E_3^{\mathrm{nd}}\) is
defined for all forward iterates. The map \(T\) is continuous on
\(\mathcal E_3^{\mathrm{nd}}\), since its components involve polynomial
operations and square roots whose denominators are nonzero there.
\end{remark}

Substitution of the centered rows into the Pearson correlation formula
gives
\begin{align}\label{eq:T-components}
T_1(a,b,c)
&=
-\frac{
a^2+ab+ac-4a-2bc+b+c+1
}{
2\sqrt{
(a^2-ab-a+b^2-b+1)
(a^2-ac-a+c^2-c+1)
}
},
\notag\\[0.5em]
T_2(a,b,c)
&=
-\frac{
ab-2ac+a+b^2+bc-4b+c+1
}{
2\sqrt{
(a^2-ab-a+b^2-b+1)
(b^2-bc-b+c^2-c+1)
}
},
\notag\\[0.5em]
T_3(a,b,c)
&=
\frac{
2ab-ac-a-bc-b-c^2+4c-1
}{
2\sqrt{
(a^2-ac-a+c^2-c+1)
(b^2-bc-b+c^2-c+1)
}
}.
\end{align}

\begin{definition}[Omega-limit set]
\label{def:omega-limit-set}
For \(p\in\mathcal E_3^{\mathrm{nd}}\), the \emph{omega-limit set} of
\(p\) is
\[
\omega(p)
=
\left\{
p_\infty\in\mathcal E_3:
\text{there exists }k_j\to\infty\text{ with }
T^{k_j}(p)\to p_\infty
\right\}.
\]
We take accumulation points in the compact elliptope \(\mathcal E_3\), so
this definition also detects a hypothetical asymptotic approach to the
excluded vertex \((1,1,1)\).
\end{definition}

\section{Symmetry Properties}\label{sec:symmetry}

Let \(S_3\) denote the symmetric group on \(\{1,2,3\}\).
The fixed-point condition \(C(A)=A\) is invariant under simultaneous
permutation of rows and columns, and this induces an \(S_3\)-action on
the coordinate triples \((a,b,c)\). We use this symmetry to organize the
fixed points and their stability.

\begin{theorem}[Permutation symmetry]
Let \(A\in\mathcal D_n\), and let \(R\) be an \(n\times n\)
permutation matrix. Then \(RAR^T\in\mathcal D_n\) and
\begin{equation}\label{eq:correlation-permutation-equivariance}
C(RAR^T)
=
RC(A)R^T.
\end{equation}
\end{theorem}

\begin{proof}
Left multiplication by \(R\) permutes the rows of \(A\), and right
multiplication by \(R^T\) applies the same coordinate permutation to
every row. These operations preserve the property that each row is
nonconstant, so \(RAR^T\in\mathcal D_n\). A common permutation of
coordinates preserves Pearson correlation. Therefore the pairwise row
correlations are permuted in the same way, yielding
equation~\eqref{eq:correlation-permutation-equivariance}.
\end{proof}

For \(\sigma\in S_3\), let \(R_\sigma\) be the corresponding
\(3\times3\) permutation matrix acting on the rows and columns of a
correlation matrix. For \(p\in\mathcal E_3\), define
\(\sigma\cdot p\) as the unique point satisfying
\begin{equation}\label{eq:coordinate-action}
A(\sigma\cdot p)
=
R_\sigma A(p)R_\sigma^T.
\end{equation}
This is well-defined because simultaneous permutation of rows and
columns preserves the correlation-matrix form and merely permutes the
three off-diagonal entries. We denote by \(K_\sigma\) the induced
\(3\times3\) permutation matrix on coordinate triples, so that
\[
\sigma\cdot p=K_\sigma p.
\]

\begin{corollary}\label{cor:T-equivariant}
For every \(\sigma\in S_3\) and
\(p\in\mathcal E_3^{\mathrm{nd}}\),
\[
T(\sigma\cdot p)
=
\sigma\cdot T(p).
\]
\end{corollary}

\begin{proof}
This follows from
equation~\eqref{eq:correlation-permutation-equivariance} by applying the
matrix identity to \(A(p)\) and comparing the three off-diagonal
entries.
\end{proof}

The expressions under the square roots in
equation~\eqref{eq:T-components} are strictly positive at every point
of \(\mathcal E_3^{\mathrm{nd}}\). Hence the same coordinate formulas
define a smooth map on an open subset of \(\mathbb R^3\) containing
\(\mathcal E_3^{\mathrm{nd}}\). We use \(DT(p)\) for the Jacobian of
this smooth extension. Permutation equivariance continues to hold on
this open set because
equation~\eqref{eq:correlation-permutation-equivariance} does not
require positive semidefiniteness.

For \(T=(T_1,T_2,T_3)\), the Jacobian matrix at
\((a,b,c)\in\mathcal E_3^{\mathrm{nd}}\) is
\begin{equation}\label{eq:jacobian-definition}
DT(a,b,c)
=
\begin{pmatrix}
\partial_aT_1&\partial_bT_1&\partial_cT_1\\
\partial_aT_2&\partial_bT_2&\partial_cT_2\\
\partial_aT_3&\partial_bT_3&\partial_cT_3
\end{pmatrix}.
\end{equation}

\begin{corollary}\label{cor:similar_jacobians}
The Jacobian matrices at permutation-equivalent fixed points are similar
and therefore have identical spectra.
\end{corollary}

\begin{proof}
By Corollary~\ref{cor:T-equivariant}, the smooth extension
satisfies
\[
T(K_\sigma p)
=
K_\sigma T(p).
\]
Differentiating at \(p\) gives
\[
DT(K_\sigma p)K_\sigma
=
K_\sigma DT(p).
\]
Since \(K_\sigma\) is invertible,
\[
DT(\sigma\cdot p)
=
K_\sigma DT(p)K_\sigma^{-1}.
\]
Thus the Jacobian matrices at \(p\) and \(\sigma\cdot p\) are similar
and have the same spectrum.
\end{proof}

We write
\[
\operatorname{Fix}(T)
=
\left\{
p\in\mathcal E_3^{\mathrm{nd}}:
T(p)=p
\right\}.
\]

For \(p\in\mathcal E_3^{\mathrm{nd}}\), write
\[
S_3\cdot p
=
\{K_\sigma p:\sigma\in S_3\}
\]
for its orbit, and
\[
\operatorname{Stab}_{S_3}(p)
=
\{\sigma\in S_3:K_\sigma p=p\}
\]
for its stabilizer.

\begin{corollary}[Orbit structure of fixed points]
\label{cor:fixed-point-orbits}
For every \(p\in\operatorname{Fix}(T)\), the orbit \(S_3\cdot p\) is
contained in \(\operatorname{Fix}(T)\). The possible orbit sizes are
\(1\), \(3\), and \(6\).
\end{corollary}

\begin{proof}
By Corollary~\ref{cor:T-equivariant},
\[
T(K_\sigma p)=K_\sigma T(p)
\]
for every \(\sigma\in S_3\). Since \(p\in\operatorname{Fix}(T)\),
\[
T(K_\sigma p)
=
K_\sigma T(p)
=
K_\sigma p.
\]
Thus
\[
S_3\cdot p
\subseteq
\operatorname{Fix}(T).
\]

By the orbit--stabilizer theorem~\cite{dummitfoote2004},
\[
|S_3\cdot p|
=
\frac{
|S_3|
}{
|\operatorname{Stab}_{S_3}(p)|
}.
\]
Since \(|S_3|=6\), the orbit size divides \(6\), and hence belongs to
\[
\{1,2,3,6\}.
\]

Suppose that the orbit has size \(2\). Then the stabilizer has order
\(3\) and contains a three-cycle. If a three-cycle fixes
\(p=(a,b,c)\), then \(a=b=c\). In this case every element of \(S_3\)
fixes \(p\), so the orbit has size \(1\), a contradiction. Thus an
orbit of size \(2\) cannot occur.

If \(a=b=c\), the orbit has size \(1\). If exactly two coordinates are
equal, the orbit has size \(3\). If the coordinates are pairwise
distinct, the orbit has size \(6\).
\end{proof}

\section{Fixed-Point Equation and Complete Classification}
\label{sec:fixed-point-equations}

\begin{lemma}[Orthogonal equivalence of full-rank Gram factors]
\label{lem:orthogonal-gram-factors}
Let \(X,Y\in\mathbb R^{m\times r}\) have full column rank. If
\[
XX^T=YY^T,
\]
then there exists an orthogonal matrix
\(Q\in\mathbb R^{r\times r}\) such that
\[
Y=XQ.
\]
\end{lemma}

\begin{proof}
For every real matrix \(M\),
\[
\ker(MM^T)=\ker(M^T).
\]
Indeed, for every vector \(z\) of the appropriate dimension,
\[
z^TMM^Tz=\|M^Tz\|^2.
\]
Thus \(MM^Tz=0\) implies \(M^Tz=0\), while the converse implication
is immediate.
The identity \(XX^T=YY^T\) therefore implies
\[
\ker(X^T)=\ker(Y^T).
\]
Taking orthogonal complements shows that \(X\) and \(Y\) have the same
column space. Since \(X\) has full column rank, there exists a matrix
\(Q\in\mathbb R^{r\times r}\) such that
\[
Y=XQ.
\]
The full column rank of \(Y\) implies that \(Q\) is invertible.

Substitution into \(YY^T=XX^T\) gives
\[
XQQ^TX^T=XX^T.
\]
Multiplying on the left by the left inverse
\[
(X^TX)^{-1}X^T
\]
and on the right by its transpose yields
\[
QQ^T=I_r.
\]
Thus \(Q\) is orthogonal.
\end{proof}
\begin{lemma}[Three eigenlines in dimension two]
\label{lem:three-eigenlines}
Let \(L:\mathbb R^2\to\mathbb R^2\) be linear. If \(L\) has three
distinct eigenlines, then \(L\) is a scalar operator.
\end{lemma}

\begin{proof}
Choose nonzero eigenvectors \(u\) and \(v\) spanning two of the
eigenlines. Since the lines are distinct, \(u\) and \(v\) form a basis
of \(\mathbb R^2\). Write
\[
Lu=\alpha u,
\qquad
Lv=\beta v.
\]
A nonzero vector on the third eigenline has the form
\[
w=su+tv,
\qquad
st\ne0.
\]
If \(Lw=\gamma w\), then
\[
s(\alpha-\gamma)u+t(\beta-\gamma)v=0.
\]
Since \(u,v\) form a basis and \(st\ne0\), it follows that
\[
\alpha=\beta=\gamma.
\]
Hence \(L=\alpha I_2\).
\end{proof}

\begin{lemma}[Invariant mixed boundary curve]
\label{lem:mixed-invariant-curve}
For
\[
-1<t<1,
\]
define
\[
\gamma_{\mathrm{mix}}(t)=(-1,-t,t).
\]
Then
\[
\gamma_{\mathrm{mix}}(t)\in\mathcal E_3^{\mathrm{nd}},
\]
and the curve is invariant under \(T\). More precisely, if
\begin{equation}\label{eq:mixed-scalar-map}
g(t)
=
\frac{4t}
{\sqrt{3t^4+10t^2+3}},
\end{equation}
then
\begin{equation}\label{eq:mixed-invariant-map}
T\bigl(\gamma_{\mathrm{mix}}(t)\bigr)
=
\gamma_{\mathrm{mix}}\bigl(g(t)\bigr),
\end{equation}
or equivalently,
\[
T(-1,-t,t)
=
\bigl(-1,-g(t),g(t)\bigr).
\]
\end{lemma}

\begin{proof}
For
\[
p=\gamma_{\mathrm{mix}}(t)=(-1,-t,t),
\]
the elliptope determinant is
\[
1+2(-1)(-t)t-(-1)^2-(-t)^2-t^2=0,
\]
so \(p\in\partial\mathcal E_3\). Since \(-1<t<1\), the associated
correlation matrix
\[
A(-1,-t,t)
\]
is not the all-ones matrix and hence belongs to
\(\mathcal C_3^{\mathrm{nd}}\).

The rows of \(A(-1,-t,t)\) are
\[
r_1=(1,-1,-t),
\qquad
r_2=(-1,1,t)=-r_1,
\qquad
r_3=(-t,t,1).
\]
Their centered versions are
\[
\widetilde r_1
=
\frac13(3+t,t-3,-2t),
\qquad
\widetilde r_2=-\widetilde r_1,
\]
and
\[
\widetilde r_3
=
\frac13(-3t-1,3t-1,2).
\]
A direct calculation gives
\[
\widetilde r_1\cdot\widetilde r_3
=
-\frac{8t}{3},
\]
together with
\[
\|\widetilde r_1\|^2
=
\frac23(t^2+3),
\qquad
\|\widetilde r_3\|^2
=
\frac23(3t^2+1).
\]
Therefore
\[
\operatorname{corr}(r_1,r_3)
=
-\frac{4t}
{\sqrt{(t^2+3)(3t^2+1)}}
=
-g(t).
\]
Since \(r_2=-r_1\),
\[
\operatorname{corr}(r_1,r_2)=-1,
\qquad
\operatorname{corr}(r_2,r_3)=g(t).
\]
Hence
\[
T(-1,-t,t)
=
\bigl(-1,-g(t),g(t)\bigr).
\]

It remains to verify that the image parameter remains in \((-1,1)\).
For \(0<|t|<1\),
\[
g(t)^2
=
\frac{16t^2}{3t^4+10t^2+3},
\]
and
\[
3t^4+10t^2+3-16t^2
=
3(t^2-1)^2>0.
\]
Thus
\[
|g(t)|<1.
\]
Also \(g(0)=0\). Hence \(g((-1,1))\subset(-1,1)\), which proves
invariance of the curve.
\end{proof}

\begin{theorem}[Complete fixed-point classification in dimension three]
\label{thm:complete-classification}
The fixed-point set of \(T\) on
\(\mathcal E_3^{\mathrm{nd}}\) is
\begin{equation}\label{eq:complete-fixed-point-set}
\begin{aligned}
\operatorname{Fix}(T)
=
\biggl\{&
\left(-\frac12,-\frac12,-\frac12\right),
(1,-1,-1),
(-1,1,-1),
(-1,-1,1),\\
&
(-1,0,0),
(0,-1,0),
(0,0,-1)
\biggr\}.
\end{aligned}
\end{equation}
In particular, the \(3\times3\) correlation map has exactly seven fixed
points on its natural nondegenerate domain.
\end{theorem}

\begin{proof}
Let
\[
A=A(a,b,c)\in\mathcal C_3\cap\mathcal D_3
\]
be a fixed point:
\[
C(A)=A.
\]
By Proposition~\ref{prop:rank-bound},
\[
\operatorname{rank}A\le2.
\]
Since \(A\) has diagonal entries equal to one, it is nonzero. Hence
\[
\operatorname{rank}A\in\{1,2\}.
\]

\medskip
\noindent
\textbf{Case 1: \(\operatorname{rank}A=1\).}

Because \(A\) is positive semidefinite and has rank one, there exists
\(v\in\mathbb R^3\) such that
\[
A=vv^T.
\]
The unit-diagonal condition implies
\[
v_i^2=A_{ii}=1,
\qquad i=1,2,3,
\]
and therefore
\[
v\in\{-1,1\}^3.
\]
The vectors \(v\) and \(-v\) produce the same matrix \(vv^T\). The
constant sign class gives the all-ones matrix, which lies outside
\(\mathcal D_3\). Hence every rank-one fixed point has one of the
following three off-diagonal coordinate triples:
\[
(1,-1,-1),
\qquad
(-1,1,-1),
\qquad
(-1,-1,1).
\]
Conversely, Proposition~\ref{prop:sign-fixed-points} shows that all
three matrices are fixed. These are therefore precisely the rank-one
fixed points.

\medskip
\noindent
\textbf{Case 2: \(\operatorname{rank}A=2\).}

Choose a full-column-rank matrix
\[
X\in\mathbb R^{3\times2}
\]
such that
\begin{equation}\label{eq:rank-two-gram-factorization}
A=XX^T.
\end{equation}
Write the rows of \(X\) as
\[
x_1^T,\quad x_2^T,\quad x_3^T,
\qquad
x_i\in\mathbb R^2.
\]
Since \(A_{ii}=1\),
\begin{equation}\label{eq:unit-gram-vectors}
\|x_i\|=1,
\qquad i=1,2,3.
\end{equation}

Let
\[
H_3
=
I_3-\frac13\mathbf1\mathbf1^T
\]
and define
\begin{equation}\label{eq:S-definition}
S=X^TH_3X.
\end{equation}
The \(i\)-th row of \(A=XX^T\) is \(x_i^TX^T\). Its centered version is
therefore
\[
x_i^TX^TH_3.
\]
Using \(H_3^T=H_3\) and \(H_3^2=H_3\), the inner product of the centered
\(i\)-th and \(j\)-th rows is
\[
x_i^TX^TH_3Xx_j
=
x_i^TSx_j,
\]
and the squared norm of the centered \(i\)-th row is
\[
x_i^TSx_i.
\]
Consequently,
\begin{equation}\label{eq:rank-two-correlation-formula}
[C(A)]_{ij}
=
\frac{x_i^TSx_j}
{\sqrt{x_i^TSx_i}\sqrt{x_j^TSx_j}}.
\end{equation}

The matrix \(S\) is symmetric and positive semidefinite because
\(H_3\) is an orthogonal projector. Set
\[
d_i=\sqrt{x_i^TSx_i}.
\]
Since \(A\in\mathcal D_3\), none of its centered rows vanishes, and thus
\begin{equation}\label{eq:positive-di}
d_i>0,
\qquad i=1,2,3.
\end{equation}
Let
\[
D=\operatorname{diag}(d_1,d_2,d_3).
\]
Equation~\eqref{eq:rank-two-correlation-formula} can be written as
\begin{equation}\label{eq:rank-two-correlation-matrix-form}
C(A)
=
D^{-1}XSX^TD^{-1}.
\end{equation}

The matrix \(S\) is positive definite. Indeed, \(S\) is positive
semidefinite, and if it were singular, then
\[
\operatorname{rank}S\le1.
\]
Equation~\eqref{eq:rank-two-correlation-matrix-form} would
then give
\[
\operatorname{rank}C(A)
\le
\operatorname{rank}S
\le1.
\]
This contradicts
\[
C(A)=A
\qquad\text{and}\qquad
\operatorname{rank}A=2.
\]
Hence
\begin{equation}\label{eq:S-positive-definite}
S\succ0.
\end{equation}
Let \(S^{1/2}\) denote the positive-definite square root of \(S\).

Define
\begin{equation}\label{eq:y-definition}
y_i
=
\frac{S^{1/2}x_i}{d_i},
\qquad
i=1,2,3,
\end{equation}
and let \(Y\in\mathbb R^{3\times2}\) be the matrix with rows \(y_i^T\).
For every \(i,j\),
\[
\langle y_i,y_j\rangle
=
\frac{x_i^TSx_j}{d_id_j}
=
[C(A)]_{ij}.
\]
Therefore
\[
YY^T=C(A)=A=XX^T.
\]
Since \(XX^T\) has rank two, both \(X\) and \(Y\) have full column rank.
Lemma~\ref{lem:orthogonal-gram-factors} gives an orthogonal matrix
\(Q\in\mathbb R^{2\times2}\) such that
\begin{equation}\label{eq:Y-XQ}
Y=XQ.
\end{equation}

The \(i\)-th row of equation~\eqref{eq:Y-XQ} is
\[
y_i^T=x_i^TQ.
\]
After transposition,
\[
y_i=Q^Tx_i.
\]
Combining this identity with equation~\eqref{eq:y-definition} gives
\[
\frac{S^{1/2}x_i}{d_i}
=
Q^Tx_i.
\]
Multiplying by \(Q\), we obtain
\begin{equation}\label{eq:eigenline-relation}
QS^{1/2}x_i=d_ix_i,
\qquad
i=1,2,3.
\end{equation}
Thus each line \(\operatorname{span}\{x_i\}\) is an eigenline of the
real linear operator
\[
L=QS^{1/2}.
\]

Because \(X\) has rank two, its row vectors span \(\mathbb R^2\).
Hence the vectors \(x_1,x_2,x_3\) cannot all lie on one line. They
therefore determine either three distinct eigenlines or exactly two
distinct eigenlines.

\medskip
\noindent
\textbf{Subcase 2a: the vectors determine three distinct eigenlines.}

By Lemma~\ref{lem:three-eigenlines}, the operator \(L=QS^{1/2}\) is
scalar:
\begin{equation}\label{eq:L-scalar}
QS^{1/2}=\lambda I_2.
\end{equation}
Since \(x_i\ne0\), equations~\eqref{eq:eigenline-relation} and
\eqref{eq:L-scalar} give
\[
\lambda=d_i
\]
for every \(i\). Hence \(\lambda>0\). Multiplying equation~\eqref{eq:L-scalar} on the
left by \(Q^T\) gives
\[
S^{1/2}=\lambda Q^T.
\]
Consequently,
\[
Q^T=\lambda^{-1}S^{1/2}.
\]
The right-hand side is symmetric positive definite. Thus \(Q^T\) is
both orthogonal and symmetric positive definite. A symmetric
orthogonal matrix has eigenvalues in \(\{-1,1\}\), whereas positive
definiteness excludes the eigenvalue \(-1\). Hence
\[
Q=I_2.
\]
It follows that
\begin{equation}\label{eq:S-isotropic}
S=\lambda^2I_2.
\end{equation}

Define
\[
\bar x
=
\frac{x_1+x_2+x_3}{3},
\qquad
z_i=x_i-\bar x,
\]
and let \(Z\in\mathbb R^{3\times2}\) have rows \(z_i^T\). Since
\[
Z=H_3X,
\]
we have
\[
Z^TZ
=
X^TH_3^TH_3X
=
X^TH_3X
=
S
=
\lambda^2I_2
\]
and
\[
Z^T\mathbf1
=
X^TH_3\mathbf1
=
0.
\]
Thus the two columns of \(\lambda^{-1}Z\) form an orthonormal basis of
the two-dimensional subspace
\[
\mathbf1^\perp
=
\{u\in\mathbb R^3:u^T\mathbf1=0\}.
\]
The matrix \(\lambda^{-2}ZZ^T\) is therefore the orthogonal projector
onto \(\mathbf1^\perp\), namely \(H_3\). Hence
\begin{equation}\label{eq:ZZT-projector}
ZZ^T=\lambda^2H_3.
\end{equation}
Reading the diagonal and off-diagonal entries of
equation~\eqref{eq:ZZT-projector} gives
\begin{equation}\label{eq:centered-equilateral-relations}
\|z_i\|^2
=
\frac{2\lambda^2}{3},
\qquad
\langle z_i,z_j\rangle
=
-\frac{\lambda^2}{3}
\quad(i\ne j).
\end{equation}

Since \(x_i=\bar x+z_i\) and \(\|x_i\|=1\),
\[
1
=
\|\bar x\|^2
+
2\langle\bar x,z_i\rangle
+
\frac{2\lambda^2}{3}
\]
for every \(i\). Thus the three numbers
\[
\langle\bar x,z_1\rangle,\quad
\langle\bar x,z_2\rangle,\quad
\langle\bar x,z_3\rangle
\]
are equal. Their sum is
\[
\left\langle
\bar x,z_1+z_2+z_3
\right\rangle
=0,
\]
so each of them is zero. Moreover,
\[
Z^TZ=\lambda^2I_2
\]
shows that \(Z\) has rank two, and therefore the vectors
\(z_1,z_2,z_3\) span \(\mathbb R^2\). Hence \(\bar x\) is orthogonal to
all of \(\mathbb R^2\), which forces
\[
\bar x=0.
\]

It follows that \(x_i=z_i\). By
equation~\eqref{eq:centered-equilateral-relations} and
\(\|x_i\|=1\),
\[
1=\frac{2\lambda^2}{3},
\qquad\text{so}\qquad
\lambda^2=\frac32.
\]
For \(i\ne j\),
\[
\langle x_i,x_j\rangle
=
-\frac{\lambda^2}{3}
=
-\frac12.
\]
Therefore
\[
(a,b,c)
=
\left(-\frac12,-\frac12,-\frac12\right).
\]

Conversely, let
\[
A_{\mathrm{eq}}
=
A\left(-\frac12,-\frac12,-\frac12\right).
\]
Each row of \(A_{\mathrm{eq}}\) has the form
\[
\left(1,-\frac12,-\frac12\right)
\]
up to permutation, and therefore has sum
\[
1-\frac12-\frac12=0.
\]
Hence every row has mean zero, so row centering leaves
\(A_{\mathrm{eq}}\) unchanged. Every row has squared norm \(3/2\), and the
inner product of any two distinct rows is \(-3/4\). Hence their
Pearson correlation is
\[
\frac{-3/4}{3/2}
=
-\frac12.
\]
Therefore
\[
C(A_{\mathrm{eq}})=A_{\mathrm{eq}},
\]
so the equicorrelation point is indeed fixed.

\medskip
\noindent
\textbf{Subcase 2b: the vectors determine exactly two eigenlines.}

Since three vectors lie on two lines, two of them are collinear. After
a simultaneous permutation of the indices, which preserves fixedness
by Corollary~\ref{cor:T-equivariant}, we may assume that \(x_1\) and
\(x_2\) are collinear. By
equation~\eqref{eq:unit-gram-vectors}, both are unit vectors, so
\[
x_2=x_1
\qquad\text{or}\qquad
x_2=-x_1.
\]

Suppose first that \(x_2=x_1\). Set
\[
t=\langle x_1,x_3\rangle.
\]
Then
\[
A=A(1,t,t).
\]
Since \(A\) has rank two, \(x_3\) is not collinear with \(x_1\), and
therefore
\[
-1<t<1.
\]
The rows of \(A\) are
\[
r_1=r_2=(1,1,t),
\qquad
r_3=(t,t,1).
\]
Their centered forms are
\[
\widetilde r_1
=
\widetilde r_2
=
\frac{1-t}{3}(1,1,-2)
\]
and
\[
\widetilde r_3
=
-\frac{1-t}{3}(1,1,-2).
\]
Thus
\[
\operatorname{corr}(r_1,r_2)=1,
\qquad
\operatorname{corr}(r_1,r_3)
=
\operatorname{corr}(r_2,r_3)
=
-1,
\]
and hence
\[
T(1,t,t)=(1,-1,-1).
\]
This cannot equal \((1,t,t)\) for \(-1<t<1\). Therefore this subcase
contains no rank-two fixed point.

Suppose now that \(x_2=-x_1\). Set
\[
t=\langle x_2,x_3\rangle.
\]
Then
\[
\langle x_1,x_3\rangle=-t
\]
and
\[
A=A(-1,-t,t),
\qquad
-1<t<1.
\]
By Lemma~\ref{lem:mixed-invariant-curve},
\[
T(-1,-t,t)
=
\bigl(-1,-g(t),g(t)\bigr),
\]
where \(g\) is defined in equation~\eqref{eq:mixed-scalar-map}.

Fixedness is equivalent to
\[
g(t)=t.
\]
The value \(t=0\) is a solution. If \(t\ne0\), division by \(t\) gives
\[
\sqrt{3t^4+10t^2+3}=4,
\]
and hence
\[
3t^4+10t^2-13=0.
\]
Factoring,
\[
(t^2-1)(3t^2+13)=0.
\]
The only real solutions are
\[
t^2=1.
\]
These endpoint values correspond to rank-one matrices and were already
classified in Case~1. They are excluded from the present rank-two
case. Thus the unique rank-two solution in this subcase is
\[
t=0,
\]
which gives
\[
(-1,0,0).
\]
By permutation equivariance, the other two points in the
\(S_3\)-orbit of \((-1,0,0)\), namely
\[
(0,-1,0)
\qquad\text{and}\qquad
(0,0,-1),
\]
are also fixed.

Every fixed point has rank one or two, and the preceding cases cover
both possibilities. Together with the direct verification of the
listed matrices, this proves
equation~\eqref{eq:complete-fixed-point-set}.
\end{proof}

\begin{remark}[Dimension dependence of the classification]
\label{rem:dimension-dependence}
Proposition~\ref{prop:rank-bound} and the classification of
nondegenerate sign-valued fixed points are valid for arbitrary
dimension. The complete classification proved above, however, uses
features specific to \(n=3\).

First, Proposition~\ref{prop:rank-bound} restricts every fixed point
to rank one or two. In the rank-two case, the associated linear
operator acts on \(\mathbb R^2\), where three distinct eigenlines force
the operator to be scalar. Moreover, the two columns of the centered Gram factor \(H_3X\) span
the two-dimensional space
\[
\mathbf1^\perp\subset\mathbb R^3.
\]
For \(n\ge4\), Proposition~\ref{prop:rank-bound} does not exclude
fixed points of ranks \(2,\ldots,n-1\), and the two-dimensional
arguments used above no longer yield a complete classification.
\end{remark}

The fixed points decompose into three permutation orbits:
\begin{equation}\label{eq:fixed-point-orbit-decomposition}
\begin{aligned}
\mathcal O_{\mathrm{eq}}
&=
\left\{
\left(-\frac12,-\frac12,-\frac12\right)
\right\},
\\
\mathcal O_{\mathrm{pat}}
&=
\{
(1,-1,-1),
(-1,1,-1),
(-1,-1,1)
\},
\\
\mathcal O_{\mathrm{mix}}
&=
\{
(-1,0,0),
(0,-1,0),
(0,0,-1)
\}.
\end{aligned}
\end{equation}

The first orbit has size \(1\), because all three coordinates are equal.
The patterned and mixed orbits each have size \(3\), because exactly two
coordinates coincide and the exceptional coordinate can occupy any of
the three positions. Hence
\[
|\mathcal O_{\mathrm{eq}}|=1,
\qquad
|\mathcal O_{\mathrm{pat}}|=3,
\qquad
|\mathcal O_{\mathrm{mix}}|=3.
\]
Thus the seven fixed points form one orbit of size one and two orbits of
size three. Although a general point with pairwise distinct coordinates
has an orbit of size \(6\), no orbit of size \(6\) occurs in the
fixed-point set.

By Corollary~\ref{cor:similar_jacobians}, all fixed points in the same
orbit have the same Jacobian spectrum. Therefore the orbit decomposition
reduces the stability analysis from seven fixed points to three orbit
representatives.

For reference in the geometry, stability, and basin analyses, we use the
labels listed in Table~\ref{tab:fixed-point-labels}.

\begin{table}[htbp]
\centering
\caption{Labels, families, and ranks of the seven analytically
classified fixed points.}
\label{tab:fixed-point-labels}
\begin{tabular}{cllc}
\toprule
Label & Family & Coordinates \((a,b,c)\) & Rank\\
\midrule
FP1 & Patterned
& \((1,-1,-1)\) & \(1\)\\
FP2 & Mixed
& \((0,0,-1)\) & \(2\)\\
FP3 & Patterned
& \((-1,1,-1)\) & \(1\)\\
FP4 & Mixed
& \((0,-1,0)\) & \(2\)\\
FP5 & Equicorrelation
& \(\left(-\frac12,-\frac12,-\frac12\right)\) & \(2\)\\
FP6 & Patterned
& \((-1,-1,1)\) & \(1\)\\
FP7 & Mixed
& \((-1,0,0)\) & \(2\)\\
\bottomrule
\end{tabular}
\end{table}

\section{Geometry of Fixed Points}\label{sec:geometry}

Figure~\ref{fig:fixed-points-3d} displays all seven fixed points
embedded in the elliptope, with marker shape and color indicating
fixed-point family and matrix rank. The elliptope boundary can be
written as
\[
c
=
ab
\pm
\sqrt{(1-a^2)(1-b^2)}.
\]

\begin{figure}[!htbp]
\centering
\includegraphics[
width=\textwidth,
height=0.55\textheight,
keepaspectratio
]{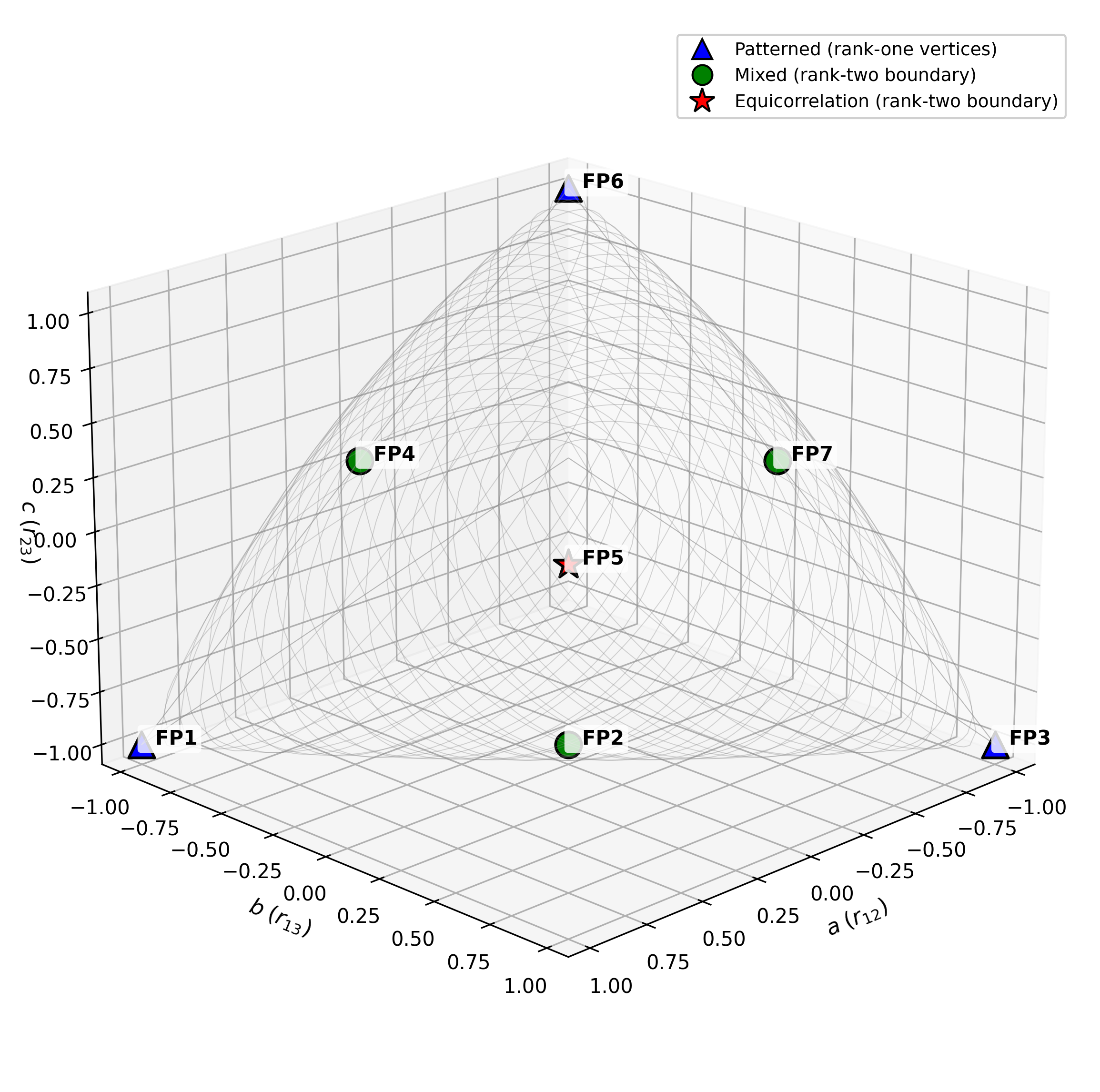}
\caption{The complete fixed-point set in the three-dimensional
elliptope. The patterned fixed points FP1 \(=(1,-1,-1)\),
FP3 \(=(-1,1,-1)\), and FP6 \(=(-1,-1,1)\), shown as blue
triangles, are rank-one vertices. The mixed points
FP2 \(=(0,0,-1)\), FP4 \(=(0,-1,0)\), and FP7 \(=(-1,0,0)\),
shown as green circles, and the equicorrelation point
FP5 \(=\left(-\frac12,-\frac12,-\frac12\right)\), shown as a red
star, are rank-two boundary points. The degenerate rank-one vertex
\((1,1,1)\) is not marked as a fixed point because the centered rows
vanish there and the correlation map is undefined.}
\label{fig:fixed-points-3d}
\end{figure}

\begin{remark}
Proposition~\ref{prop:rank-bound} explains why every fixed point lies on
the boundary. The complete classification exhibits two matrix-rank
strata: the patterned points are rank-one vertices, whereas the mixed
points and the equicorrelation point are rank-two boundary points. The relative Lyapunov stability classification below shows that only
the patterned rank-one vertices are locally asymptotically stable
among the seven fixed points. Geometric aspects and
rank-reducibility of correlation matrices are discussed, for example,
in \cite{shapiro1982}.
\end{remark}

\subsection{Relation to Chen's Stationary Patterns}

Chen~\cite[p.~14]{chen2002generalized} states that his Figure~5 gives
all stationary points of the iterated correlation map for \(p=3\) and
\(p=4\). For \(p=3\), the displayed stationary configurations fall
into three structural types: the symmetric equicorrelation pattern
\(3(1)\), the mixed pattern \(3(2)\), and the rank-one sign-valued
pattern \(3(3)\); see also
\cite[pp.~16--17]{chen2002generalized}.

In the coordinate representation used here, these three types
correspond respectively to
\[
\left(-\frac12,-\frac12,-\frac12\right),
\]
the three mixed points
\[
(-1,0,0),\qquad
(0,-1,0),\qquad
(0,0,-1),
\]
and the three patterned points
\[
(1,-1,-1),\qquad
(-1,1,-1),\qquad
(-1,-1,1).
\]
Thus the fixed-point set obtained in
Theorem~\ref{thm:complete-classification} agrees with Chen's
\(p=3\) stationary patterns. The theorem provides a self-contained
analytical proof in elliptope coordinates that these seven
configurations exhaust the fixed-point set of \(T\) on
\(\mathcal E_3^{\mathrm{nd}}\).

\subsection{A Sign-Pattern Selection Phenomenon}

Among the eight triples in \(\{-1,1\}^3\), the
positive-semidefiniteness condition in
equation~\eqref{eq:elliptope-e3} is satisfied precisely by
\[
(1,1,1),
\qquad
(1,-1,-1),
\qquad
(-1,1,-1),
\qquad
(-1,-1,1).
\]
Indeed, for a sign triple,
\[
1+2abc-a^2-b^2-c^2
=
2(abc-1),
\]
so positive semidefiniteness is equivalent to \(abc=1\).

The point \((1,1,1)\) is excluded from the domain of \(T\), because
all three centered rows vanish. Consequently, the only nondegenerate
positive-semidefinite sign triples are
\[
(1,-1,-1),
\qquad
(-1,1,-1),
\qquad
(-1,-1,1).
\]

The remaining four sign triples,
\[
(-1,-1,-1),
\qquad
(-1,1,1),
\qquad
(1,-1,1),
\qquad
(1,1,-1),
\]
fail the determinant condition. This is the concrete \(n=3\)
realization of Proposition~\ref{prop:sign-fixed-points}.

\subsection{The Degenerate Case}

The point \((1,1,1)\) corresponds to a matrix whose centered rows
vanish identically. Consequently, the Pearson correlation coefficients
are undefined, and the point is excluded from the domain of \(T\).

\section{Relative Lyapunov Stability Analysis}
\label{sec:local-stability}

For a square matrix \(M\in\mathbb R^{m\times m}\), we write
\[
\operatorname{spec}(M)
\]
for its spectrum,
\[
\chi_M(\lambda)=\det(\lambda I_m-M)
\]
for its characteristic polynomial, and
\[
\rho(M)
=
\max\{|\lambda|:\lambda\in\operatorname{spec}(M)\}
\]
for its spectral radius.

\begin{definition}
A fixed point \(p\in\mathcal E_3^{\mathrm{nd}}\) is
\emph{Lyapunov stable relative to \(\mathcal E_3^{\mathrm{nd}}\)} if,
for every relative neighborhood \(V\) of \(p\), there exists a relative
neighborhood \(U\) of \(p\) such that every orbit starting in \(U\) is
defined for all forward iterates and remains in \(V\).

It is \emph{locally asymptotically stable relative to
\(\mathcal E_3^{\mathrm{nd}}\)} if it is Lyapunov stable and every orbit
from some relative neighborhood converges to \(p\). A fixed point that
is not Lyapunov stable is called \emph{unstable}.
\end{definition}

We study relative Lyapunov stability using the Jacobian matrix defined in
equation~\eqref{eq:jacobian-definition}. At each fixed point listed in
Theorem~\ref{thm:complete-classification}, the denominators in
equation~\eqref{eq:T-components} are nonzero. Consequently, the
coordinate formulas for \(T\) define a \(C^1\) map on an open
neighborhood of that point in \(\mathbb R^3\).

The standard linearization criterion for discrete dynamical systems
implies that a fixed point \(p\) is locally asymptotically stable if
every eigenvalue of \(DT(p)\) has modulus strictly less than \(1\); see
\cite{hale2009dynamics,katok1995introduction}. This criterion applies
directly to the patterned fixed points.

The mixed and equicorrelation points lie on the boundary of the
elliptope, so an expanding ambient eigenvalue alone does not imply
instability relative to \(\mathcal E_3^{\mathrm{nd}}\). We therefore
exhibit invariant curves contained in the elliptope boundary along
which perturbations within the domain move away from the corresponding
fixed points.

By Corollary~\ref{cor:similar_jacobians}, it is sufficient to evaluate
the Jacobian at
\[
p_{\mathrm{pat}}
=
(-1,-1,1),
\qquad
p_{\mathrm{mix}}
=
(-1,0,0),
\qquad
p_{\mathrm{eq}}
=
\left(-\frac12,-\frac12,-\frac12\right).
\]

\begin{theorem}[Complete relative Lyapunov stability classification]
\label{thm:stability}
Relative to \(\mathcal E_3^{\mathrm{nd}}\), the three points of
\(\mathcal O_{\mathrm{pat}}\) are locally asymptotically stable, whereas
the four points of
\(\mathcal O_{\mathrm{eq}}\cup\mathcal O_{\mathrm{mix}}\)
are Lyapunov unstable.
\end{theorem}

\begin{proof}
At the patterned representative, direct differentiation of
equation~\eqref{eq:T-components} gives
\begin{equation}\label{eq:jacobian-patterned}
DT(p_{\mathrm{pat}})
=
\begin{pmatrix}
0&0&0\\
0&0&0\\
0&0&0
\end{pmatrix},
\qquad
\chi_{\mathrm{pat}}(\lambda)
:=
\chi_{DT(p_{\mathrm{pat}})}(\lambda)
=
\lambda^3.
\end{equation}
Equation~\eqref{eq:jacobian-patterned} gives
\[
\rho(DT(p_{\mathrm{pat}}))
=
0<1.
\]
The ambient extension is locally asymptotically stable at
\(p_{\mathrm{pat}}\). It agrees with \(T\) on the elliptope, and a
sufficiently small neighborhood of \(p_{\mathrm{pat}}\) is disjoint
from the excluded vertex \((1,1,1)\). Therefore
\(p_{\mathrm{pat}}\) is locally asymptotically stable relative to
\(\mathcal E_3^{\mathrm{nd}}\).

At the mixed representative, direct differentiation of
equation~\eqref{eq:T-components} gives
\begin{equation}\label{eq:mixed-linearization-data}
\begin{gathered}
DT(p_{\mathrm{mix}})
=
\begin{pmatrix}
0&0&0\\[0.5ex]
-\frac{\sqrt3}{6}&\frac{5\sqrt3}{6}&-\frac{\sqrt3}{2}\\[0.5ex]
-\frac{\sqrt3}{6}&-\frac{\sqrt3}{2}&\frac{5\sqrt3}{6}
\end{pmatrix},\\[1ex]
\chi_{\mathrm{mix}}(\lambda)
:=
\chi_{DT(p_{\mathrm{mix}})}(\lambda)
=
\frac{\lambda(3\lambda^2-5\sqrt3\,\lambda+4)}{3},
\qquad
\operatorname{spec}\bigl(DT(p_{\mathrm{mix}})\bigr)
=
\left\{
0,\frac{\sqrt3}{3},\frac{4\sqrt3}{3}
\right\}.
\end{gathered}
\end{equation}

The spectrum in equation~\eqref{eq:mixed-linearization-data} contains
the expanding eigenvalue \(4\sqrt3/3\). This expansion is realized by
an admissible direction in the elliptope. Consider
\[
\gamma_{\mathrm{mix}}(t)
=
(-1,-t,t).
\]
For \(|t|<1\),
\[
\gamma_{\mathrm{mix}}(t)
\in
\mathcal E_3^{\mathrm{nd}},
\]
and the determinant condition holds with equality.
Lemma~\ref{lem:mixed-invariant-curve} gives
\[
T(\gamma_{\mathrm{mix}}(t))
=
\gamma_{\mathrm{mix}}(g(t)).
\]
Moreover,
\[
g(0)=0,
\qquad
g'(0)
=
\frac4{\sqrt3}
=
\frac{4\sqrt3}{3}
>1.
\]

By continuity of \(g'\) and the mean value theorem, there exist
\(\eta>0\) and \(\mu>1\) such that
\[
|g(t)|\ge \mu|t|
\qquad
\text{whenever }0<|t|<\eta.
\]
Shrinking \(\eta\) if necessary, assume also that \(\eta<1\), so that
\(\gamma_{\mathrm{mix}}(t)\in\mathcal E_3^{\mathrm{nd}}\) throughout this
parameter interval. Fix
\[
\varepsilon_0=\frac{\eta}{2}.
\]
Let \(U\) be any relative neighborhood of \(p_{\mathrm{mix}}\) in
\(\mathcal E_3^{\mathrm{nd}}\). Since
\(\gamma_{\mathrm{mix}}(t)\to p_{\mathrm{mix}}\) as \(t\to0\), one may choose
\(t_0\ne0\) so small that
\[
\gamma_{\mathrm{mix}}(t_0)\in U,
\qquad
0<|t_0|<\varepsilon_0.
\]
Set \(t_{j+1}=g(t_j)\). As long as \(|t_j|<\varepsilon_0\), one has
\[
|t_{j+1}|
\ge
\mu |t_j|.
\]
Hence, after finitely many iterates,
\[
|t_j|\ge\varepsilon_0.
\]
Since
\[
\|\gamma_{\mathrm{mix}}(t_j)-p_{\mathrm{mix}}\|
=
\sqrt2\,|t_j|,
\]
the orbit leaves the fixed relative neighborhood
\[
B_{\sqrt2\,\varepsilon_0}(p_{\mathrm{mix}})
\cap
\mathcal E_3^{\mathrm{nd}}.
\]
Thus arbitrarily close admissible initial conditions leave a fixed
neighborhood of \(p_{\mathrm{mix}}\), proving Lyapunov instability relative
to \(\mathcal E_3^{\mathrm{nd}}\).

At the equicorrelation representative, direct differentiation of equation~\eqref{eq:T-components} gives
\begin{equation}\label{eq:equicorrelation-linearization-data}
\begin{gathered}
DT(p_{\mathrm{eq}})
=
\begin{pmatrix}
1&-\frac12&-\frac12\\
-\frac12&1&-\frac12\\
-\frac12&-\frac12&1
\end{pmatrix},\\[1ex]
\chi_{\mathrm{eq}}(\lambda)
=
\frac{\lambda(2\lambda-3)^2}{4},
\qquad
\operatorname{spec}\bigl(DT(p_{\mathrm{eq}})\bigr)
=
\left\{
0,\frac32,\frac32
\right\}.
\end{gathered}
\end{equation}

The spectrum in
equation~\eqref{eq:equicorrelation-linearization-data} contains the
expanding eigenvalue \(3/2\). Consider the boundary curve
\[
\gamma_{\mathrm{eq}}(x)
=
\left(
x,x,2x^2-1
\right).
\]
For \(x\) near \(-\frac12\), this curve belongs to
\(\mathcal E_3^{\mathrm{nd}}\), and the determinant condition holds
with equality. Direct substitution gives
\begin{equation}\label{eq:equicorrelation-invariant-map}
T(\gamma_{\mathrm{eq}}(x))
=
\gamma_{\mathrm{eq}}(f(x)),
\qquad
f(x)
=
\frac{x}
{\sqrt{4x^2+6x+3}}.
\end{equation}
Equation~\eqref{eq:equicorrelation-invariant-map} gives, at the
equicorrelation fixed point,
\[
f\left(-\frac12\right)
=
-\frac12,
\qquad
f'\left(-\frac12\right)
=
\frac32
>1.
\]

By continuity of \(f'\) and the mean value theorem, there exist
\(\eta>0\) and \(\mu>1\) such that
\[
\left|f(x)+\frac12\right|
\ge
\mu\left|x+\frac12\right|
\]
whenever
\[
0<\left|x+\frac12\right|<\eta.
\]
Shrinking \(\eta\) if necessary, assume that
\(\gamma_{\mathrm{eq}}(x)\in\mathcal E_3^{\mathrm{nd}}\) whenever
\(\left|x+\frac12\right|<\eta\). Fix
\[
\varepsilon_0=\frac{\eta}{2}.
\]
Let \(U\) be any relative neighborhood of \(p_{\mathrm{eq}}\) in
\(\mathcal E_3^{\mathrm{nd}}\). Since
\(\gamma_{\mathrm{eq}}(x)\to p_{\mathrm{eq}}\) as
\(x\to-\frac12\), choose \(x_0\ne-\frac12\) sufficiently close to
\(-\frac12\) that
\[
\gamma_{\mathrm{eq}}(x_0)\in U,
\qquad
0<\left|x_0+\frac12\right|<\varepsilon_0.
\]
Set \(x_{j+1}=f(x_j)\). As long as
\(\left|x_j+\frac12\right|<\varepsilon_0\),
\[
\left|x_{j+1}+\frac12\right|
\ge
\mu\left|x_j+\frac12\right|.
\]
Hence, after finitely many iterates,
\[
\left|x_j+\frac12\right|
\ge
\varepsilon_0.
\]
The first two coordinates of
\(\gamma_{\mathrm{eq}}(x_j)-p_{\mathrm{eq}}\) are both
\(x_j+\frac12\), so
\[
\|\gamma_{\mathrm{eq}}(x_j)-p_{\mathrm{eq}}\|
\ge
\sqrt2\left|x_j+\frac12\right|
\ge
\sqrt2\,\varepsilon_0.
\]
Thus the orbit leaves the fixed relative neighborhood
\[
B_{\sqrt2\,\varepsilon_0}(p_{\mathrm{eq}})
\cap
\mathcal E_3^{\mathrm{nd}}.
\]
Therefore arbitrarily close admissible initial conditions leave a fixed
neighborhood of \(p_{\mathrm{eq}}\), proving Lyapunov instability relative
to \(\mathcal E_3^{\mathrm{nd}}\).

The remaining points in the three orbits are obtained from the
representatives by permutation symmetry. Since coordinate permutations
preserve \(\mathcal E_3^{\mathrm{nd}}\) and \(T\) is equivariant under
the \(S_3\)-action, they carry relative neighborhoods, forward orbits,
and invariant boundary curves to their corresponding permuted objects.
Therefore local asymptotic stability and Lyapunov instability are
constant on each permutation orbit. This transfers the three
conclusions to every member of the corresponding orbit.
\end{proof}

Combining Theorems~\ref{thm:complete-classification}
and~\ref{thm:stability}, the locally asymptotically stable fixed points
of \(T\) are precisely
\[
(1,-1,-1),
\qquad
(-1,1,-1),
\qquad
(-1,-1,1).
\]
The four points in
\(\mathcal O_{\mathrm{eq}}\cup\mathcal O_{\mathrm{mix}}\)
are unstable.

\begin{remark}\label{rem:local-not-global}
Theorems~\ref{thm:complete-classification}
and~\ref{thm:stability} are local and stationary statements. Forward
well-posedness follows separately from Theorem~\ref{thm:forward-invariance}.
Neither the fixed-point classification nor the relative Lyapunov stability
classification alone excludes nontrivial \(\omega\)-limit sets. The global
convergence question is resolved in Section~\ref{sec:global-convergence} by a
projective kernel analysis of rank-two iterates.
\end{remark}

\begin{remark}\label{rem:symbolic-jacobians}
The Jacobian matrices, characteristic polynomials, eigenvalues, and
identities for the invariant curves above are obtained directly from the displayed
coordinate formulas. These exact calculations were also checked symbolically
with SymPy as part of the computational validation.
The spectral radii of the patterned, mixed, and equicorrelation
representatives are
\[
0,
\qquad
\frac{4\sqrt3}{3},
\qquad
\frac32,
\]
respectively.
\end{remark}

\section{Global Convergence in Dimension Three}
\label{sec:global-convergence}

The rank identity reduces the global problem sharply in dimension three.
For every \(P\in\mathcal C_3^{\mathrm{nd}}\),
\[
\operatorname{rank}C(P)
=
\operatorname{rank}(PH_3)
\le2.
\]
Thus every admissible matrix trajectory enters the singular boundary of the
elliptope after one iteration. The remaining rank-two dynamics can be
encoded by the one-dimensional kernel of the current correlation matrix.
We first make this coordinate precise.

\begin{definition}[Projective kernel coordinate and normalization factors]
\label{def:projective-kernel-coordinate}
Let \(\mathbb P^2(\mathbb R)\) denote the real projective plane, that is,
the set of one-dimensional linear subspaces of \(\mathbb R^3\). For
\(q\in\mathbb R^3\setminus\{0\}\), write
\[
[q]=\operatorname{span}\{q\}\in\mathbb P^2(\mathbb R)
\]
for the corresponding projective point.

Let \(P\in\mathcal C_3^{\mathrm{nd}}\) have rank two. Since
\(\ker P\) is one-dimensional, its \emph{projective kernel coordinate} is
\[
\kappa(P)=[q],
\]
where \(q\ne0\) is any vector spanning \(\ker P\). For such a
representative define
\[
s(q)=q^T\mathbf1=q_1+q_2+q_3.
\]
For every \(P\in\mathcal C_3^{\mathrm{nd}}\), define the positive
Pearson normalization factors
\[
d_i(P)=\|P_{i,:}H_3\|,
\qquad i=1,2,3,
\]
and the diagonal matrix
\[
D(P)=\operatorname{diag}\bigl(d_1(P),d_2(P),d_3(P)\bigr).
\]
The matrix \(D(P)\) is invertible because
\(P\in\mathcal C_3^{\mathrm{nd}}\).
\end{definition}

\begin{lemma}[Kernel reconstruction on the rank-two boundary]
\label{lem:kernel-reconstruction}
Let \((a,b,c)\in\mathcal E_3^{\mathrm{nd}}\), set
\(P=A(a,b,c)\), and assume that \(P\) has rank two. Let
\(q=(q_1,q_2,q_3)^T\ne0\) span \(\ker P\). If
\(q_1q_2q_3\ne0\), then
\begin{align}
 a&=\frac{q_3^2-q_1^2-q_2^2}{2q_1q_2},
 \label{eq:kernel-reconstruct-a}\\
 b&=\frac{q_2^2-q_1^2-q_3^2}{2q_1q_3},
 \label{eq:kernel-reconstruct-b}\\
 c&=\frac{q_1^2-q_2^2-q_3^2}{2q_2q_3}.
 \label{eq:kernel-reconstruct-c}
\end{align}
Consequently, on the locus \(q_1q_2q_3\ne0\), the projective kernel
coordinate \(\kappa(P)=[q]\) determines \(P\) uniquely.
\end{lemma}

\begin{proof}
Choose a full-column-rank Gram factor
\[
P=XX^T,
\qquad
X\in\mathbb R^{3\times2},
\]
whose rows are unit vectors \(x_1^T,x_2^T,x_3^T\). Since
\(\ker P=\ker X^T\),
\[
q_1x_1+q_2x_2+q_3x_3=0.
\]
Taking squared norms in
\(q_1x_1+q_2x_2=-q_3x_3\) gives
\[
q_1^2+q_2^2+2q_1q_2\langle x_1,x_2\rangle=q_3^2.
\]
Because \(a=\langle x_1,x_2\rangle\), this yields
\eqref{eq:kernel-reconstruct-a}. The other two identities follow by cyclic
permutation. Each right-hand side is homogeneous of degree zero in \(q\),
so the reconstructed matrix depends only on \([q]\).
\end{proof}

\begin{proposition}[Exact projective kernel dynamics]
\label{prop:projective-kernel-dynamics}
Let \(P\in\mathcal C_3^{\mathrm{nd}}\) have rank two, let
\(q=(q_1,q_2,q_3)^T\ne0\) span \(\ker P\), and use the notation of
Definition~\ref{def:projective-kernel-coordinate}. Choose any
full-column-rank Gram factor
\[
P=XX^T,
\qquad
X\in\mathbb R^{3\times2},
\]
and set
\[
S=X^TH_3X.
\]
Then the following statements hold.

\begin{enumerate}
\item \(s(q)=0\) if and only if \(\operatorname{rank}C(P)=1\). In this
case \(C(P)\) is one of the three nondegenerate rank-one patterned fixed
points.

\item If \(s(q)\ne0\), then \(S\succ0\),
\(\operatorname{rank}C(P)=2\), and
\begin{equation}\label{eq:kernel-update}
\kappa(C(P))=[D(P)q].
\end{equation}
Equivalently, if \(q^+\) denotes the representative of
\(\ker C(P)\) chosen by this rule, then
\[
q^+=D(P)q.
\]

\item If \(s(q)\ne0\), then for every \(i,j\in\{1,2,3\}\),
\begin{equation}\label{eq:kernel-order-identity}
 d_i(P)^2-d_j(P)^2
 =
 \frac{2\det(S)}{s(q)}(q_i-q_j).
\end{equation}
The scalar \(\det(S)\) is independent of the chosen full-column-rank Gram
factor \(X\).

\item After replacing \(q\) by \(-q\) if necessary so that \(s(q)>0\),
\begin{equation}\label{eq:kernel-order-equivalence}
q_i>q_j
\quad\Longleftrightarrow\quad
d_i(P)>d_j(P).
\end{equation}
In particular, if \(q_i=q_j\), then the equality of these two kernel
coordinates is preserved by the next rank-two iterate.
\end{enumerate}
\end{proposition}

\begin{proof}
By Proposition~\ref{prop:rank-bound},
\[
C(P)=D(P)^{-1}PH_3PD(P)^{-1}.
\]
Since \(Pq=0\),
\[
C(P)D(P)q
=
D(P)^{-1}PH_3Pq
=0.
\]
Thus \(D(P)q\in\ker C(P)\).

We first characterize the possible rank drop. If \(s(q)=q^T\mathbf1=0\),
then \(H_3q=q\), and hence
\[
PH_3q=Pq=0.
\]
Also \(PH_3\mathbf1=0\). The vectors \(q\) and \(\mathbf1\) are
linearly independent, so \(\dim\ker(PH_3)\ge2\) and therefore
\[
\operatorname{rank}(PH_3)\le1.
\]
The exact rank identity gives \(\operatorname{rank}C(P)\le1\). Since
\(C(P)\) is a correlation matrix, its rank is at least one. Hence
\(\operatorname{rank}C(P)=1\), and forward invariance together with
Corollary~\ref{cor:rank-one-correlation-states} identifies \(C(P)\) as a
nondegenerate patterned fixed point.

Conversely, assume \(s(q)\ne0\). The matrix \(S=X^TH_3X\) is positive
semidefinite. If it were singular, there would exist \(v\ne0\) such that
\(H_3Xv=0\). Hence \(Xv=c\mathbf1\) for some scalar \(c\). If \(c=0\),
then \(Xv=0\), which contradicts the full column rank of \(X\) and
\(v\ne0\). Thus \(c\ne0\). Since \(X^Tq=0\),
\[
0=q^TXv=c\,q^T\mathbf1.
\]
Because \(c\ne0\), this gives \(q^T\mathbf1=0\), contradicting
\(s(q)\ne0\). Thus
\[
S\succ0.
\]
In particular, \(H_3X\) has rank two. Since
\[
PH_3=X(X^TH_3)
\]
and \(X\) has full column rank, the linear map
\(u\mapsto Xu\) is injective on \(\mathbb R^2\). Hence left
multiplication by \(X\) does not reduce the rank of the \(2\times3\)
matrix \(X^TH_3\), and therefore
\[
\operatorname{rank}(PH_3)
=
\operatorname{rank}(X^TH_3)
=
\operatorname{rank}(H_3X)
=2.
\]
Hence \(\operatorname{rank}C(P)=2\). Its kernel is therefore
one-dimensional, and the already established inclusion
\(D(P)q\in\ker C(P)\) proves
\eqref{eq:kernel-update}. This also proves the converse implication in
part~1.

It remains to prove the order identity. Let
\[
\bar x=\frac{x_1+x_2+x_3}{3},
\qquad
y_i=x_i-\bar x,
\]
and let \(Y=H_3X\), whose rows are \(y_i^T\). Since
\(S=Y^TY\succ0\), the two columns of \(Y\) form a basis of
\(\mathbf1^\perp\). Therefore
\begin{equation}\label{eq:centered-projector-identity}
YS^{-1}Y^T=H_3,
\end{equation}
which gives
\begin{equation}\label{eq:y-Sinv-y}
y_i^TS^{-1}y_j=\delta_{ij}-\frac13.
\end{equation}
The kernel relation \(\sum_{j=1}^3q_jx_j=0\) yields
\[
\bar x=-\frac1{s(q)}\sum_{j=1}^3q_jy_j.
\]
Using \eqref{eq:y-Sinv-y},
\[
y_i^TS^{-1}\bar x
=
\frac13-\frac{q_i}{s(q)},
\]
and
\[
\bar x^TS^{-1}\bar x
=
\frac{\|q\|^2}{s(q)^2}-\frac13.
\]
Consequently,
\begin{equation}\label{eq:xi-Sinv-xi}
x_i^TS^{-1}x_i
=
1+\frac{\|q\|^2}{s(q)^2}-\frac{2q_i}{s(q)}.
\end{equation}

The centered \(i\)-th row of \(P=XX^T\) is
\(x_i^TX^TH_3\), so
\[
d_i(P)^2
=
\|P_{i,:}H_3\|^2
=
x_i^TSx_i.
\]
For the nonsingular \(2\times2\) matrix \(S\), the Cayley--Hamilton
identity gives
\[
S=(\operatorname{tr}S)I_2-\det(S)S^{-1}.
\]
Since \(\|x_i\|=1\),
\[
d_i(P)^2
=
\operatorname{tr}S-\det(S)x_i^TS^{-1}x_i.
\]
Subtracting the formulas for \(i\) and \(j\) and using
\eqref{eq:xi-Sinv-xi} yields
\eqref{eq:kernel-order-identity}.

If \(P=X'X'^T\) is another full-column-rank Gram factorization, then
Lemma~\ref{lem:orthogonal-gram-factors} gives \(X'=XQ\) for some
orthogonal \(Q\). Thus
\[
X'^TH_3X'=Q^TSQ,
\]
so \(\det(S)\) is independent of the chosen Gram factor. Finally, when
\(s(q)>0\), positivity of \(S\) gives \(\det(S)>0\), and
\eqref{eq:kernel-order-equivalence} follows from
\eqref{eq:kernel-order-identity}. If \(q_i=q_j\), then
\(d_i(P)=d_j(P)\); equation \eqref{eq:kernel-update} therefore preserves
that equality whenever the next iterate has rank two.
\end{proof}

\begin{lemma}[Triangle constraints for rank-two kernel coordinates]
\label{lem:kernel-triangle-constraints}
Let \(P=XX^T\in\mathcal C_3^{\mathrm{nd}}\) have rank two, where
\(X\in\mathbb R^{3\times2}\) has unit row vectors
\(x_1^T,x_2^T,x_3^T\). Let
\(q=(q_1,q_2,q_3)^T\ne0\) span \(\ker P\). Then, for every permutation
\((i,j,\ell)\) of \((1,2,3)\),
\begin{equation}\label{eq:kernel-triangle-inequality}
|q_i|\le |q_j|+|q_\ell|.
\end{equation}
If \(q_1q_2q_3\ne0\), all three inequalities are strict. If one
coordinate vanishes, say \(q_\ell=0\), then the other two coordinates are
nonzero and satisfy \(|q_i|=|q_j|\).
\end{lemma}

\begin{proof}
Since \(q\in\ker X^T\),
\[
q_1x_1+q_2x_2+q_3x_3=0.
\]
Thus, for every permutation \((i,j,\ell)\),
\[
|q_i|
=\|q_i x_i\|
=\|q_jx_j+q_\ell x_\ell\|
\le |q_j|+|q_\ell|,
\]
which proves \eqref{eq:kernel-triangle-inequality}. If all three
coordinates are nonzero and equality holds for one index, equality in the
Euclidean triangle inequality forces \(x_j\) and \(x_\ell\) to be
collinear in the directions required by the coefficients. The relation above
then forces \(x_i\) to be collinear with them as well, contradicting
\(\operatorname{rank}X=2\). Hence all three inequalities are strict.

If \(q_\ell=0\), then
\(q_ix_i=-q_jx_j\). Neither \(q_i\) nor \(q_j\) can vanish, since two
zero coordinates would force the remaining nonzero multiple of a unit
vector to equal zero. Taking norms gives \(|q_i|=|q_j|\).
\end{proof}

\begin{theorem}[Global convergence in dimension three]
\label{thm:global-convergence}
For every
\[
P_0\in\mathcal C_3^{\mathrm{nd}},
\]
the forward orbit
\[
P_{k+1}=C(P_k),
\qquad k\ge0,
\]
is defined for all \(k\ge0\) and converges to one of the seven fixed
correlation matrices corresponding, under the parametrization
\eqref{eq:correlation-parametrization}, to the seven coordinate points in
\eqref{eq:complete-fixed-point-set}.
\end{theorem}

\begin{proof}
Forward well-posedness follows from Theorem~\ref{thm:forward-invariance}.
The exact rank identity gives
\[
\operatorname{rank}P_1\le2.
\]
If any \(P_k\) has rank one, then
Corollary~\ref{cor:rank-one-correlation-states} shows that it is one of the
three patterned fixed points, and the orbit is constant thereafter. It therefore remains only to
consider an orbit for which
\begin{equation}\label{eq:permanent-rank-two}
\operatorname{rank}P_k=2
\qquad(k\ge1).
\end{equation}

For each \(k\ge1\), choose a nonzero vector \(q_k\) spanning
\(\ker P_k\), and set
\[
d_{i,k}=d_i(P_k),
\qquad
D_k=D(P_k)=\operatorname{diag}(d_{1,k},d_{2,k},d_{3,k}).
\]
Under \eqref{eq:permanent-rank-two},
Proposition~\ref{prop:projective-kernel-dynamics} gives
\(s(q_k)\ne0\) for every \(k\). Choose the representative \(q_1\) so
that \(s(q_1)>0\), and then choose successive representatives by
\begin{equation}\label{eq:successive-kernel-representatives}
q_{k+1}=D_kq_k.
\end{equation}
Because every diagonal entry of \(D_k\) is positive, the sign of each
kernel coordinate is preserved along the rank-two orbit.

For each \(k\ge1\), choose a full-column-rank Gram factor
\[
P_k=X_kX_k^T,
\qquad
X_k\in\mathbb R^{3\times2},
\]
and denote its unit row vectors by
\(x_{1,k}^T,x_{2,k}^T,x_{3,k}^T\). These vectors are introduced only to
invoke Lemma~\ref{lem:kernel-triangle-constraints}; the projective kernel
coordinate itself is independent of the chosen Gram factor.

The relabeling of kernel coordinates used below is compatible with the
permutation symmetry. Indeed, if
\[
\widetilde P=R_\sigma P R_\sigma^T
\]
and \(Pq=0\), then
\[
\widetilde P(R_\sigma q)
=
R_\sigma Pq
=
0.
\]
Thus simultaneous permutation of the matrix indices induces the same
permutation of the kernel coordinates.

Suppose first that all three coordinates of \(q_1\) are nonzero.
Lemma~\ref{lem:kernel-triangle-constraints} gives strict triangle
inequalities for their absolute values. Because \(s(q_1)>0\), the only
possible sign patterns are: all three coordinates positive, or exactly one
coordinate negative. Equation
\eqref{eq:successive-kernel-representatives} preserves the sign of every
coordinate. In the all-positive chamber the coordinate sum is automatically
positive; in the one-negative chamber the strict triangle inequality says
that the magnitude of the negative coordinate is smaller than the sum of the
two positive coordinates. Hence \(s(q_k)>0\) for every subsequent
rank-two iterate. Proposition~\ref{prop:projective-kernel-dynamics}
therefore applies with the same orientation at every step.

\medskip
\noindent
\textbf{Case 1: all kernel coordinates are positive.}

By permutation equivariance, relabel the indices once so that
\[
q_{1,1}\ge q_{2,1}\ge q_{3,1}>0.
\]
Proposition~\ref{prop:projective-kernel-dynamics} and the update
\(q_{k+1}=D_kq_k\) preserve this order. Normalize projectively by the first
coordinate and write
\[
u_k=\frac{q_{2,k}}{q_{1,k}},
\qquad
v_k=\frac{q_{3,k}}{q_{1,k}}.
\]
The strict triangle inequalities give
\begin{equation}\label{eq:positive-kernel-region}
1\ge u_k\ge v_k>0,
\qquad
u_k+v_k>1.
\end{equation}
Moreover,
\[
u_{k+1}
=
 u_k\frac{d_{2,k}}{d_{1,k}}
\le u_k,
\qquad
v_{k+1}
=
 v_k\frac{d_{3,k}}{d_{1,k}}
\le v_k.
\]
Thus
\[
u_k\downarrow u_\infty,
\qquad
v_k\downarrow v_\infty,
\qquad
u_\infty+v_\infty\ge1.
\]

If \(q_{1,1}=q_{2,1}=q_{3,1}\), Lemma~\ref{lem:kernel-reconstruction} gives
\[
P_1=
A\left(-\frac12,-\frac12,-\frac12\right),
\]
so the orbit is already at the equicorrelation fixed point.

Suppose next that \(q_{1,1}=q_{2,1}>q_{3,1}\). Equality of the first two
kernel coordinates is preserved. Define the projective ratio
\[
z_k=\frac{q_{3,k}}{q_{1,k}}.
\]
Then \(0<z_k<1\), the projective kernel class is represented by
\((1,1,z_k)^T\), and
\[
z_{k+1}
=
z_k\frac{d_{3,k}}{d_{1,k}}
<z_k.
\]
Hence \(z_k\downarrow z_\infty\ge0\). If \(z_\infty>0\), then \(0<z_\infty<1\), and the reconstruction
formulas show that \(P_k\) converges to the rank-two correlation matrix
\(P_\infty\) reconstructed from \(q_\infty=(1,1,z_\infty)^T\). Moreover,
\(z_{k+1}/z_k\to1\) gives \(d_{3,k}/d_{1,k}\to1\). Since each function
\(P\mapsto d_i(P)=\|P_{i,:}H_3\|\) is continuous,
\(d_3(P_\infty)=d_1(P_\infty)\). Applying
\eqref{eq:kernel-order-identity} to \((P_\infty,q_\infty)\) gives
\(q_{3,\infty}=q_{1,\infty}\), contradicting \(z_\infty<1\). Therefore
\(z_\infty=0\).
Equations \eqref{eq:kernel-reconstruct-a}--\eqref{eq:kernel-reconstruct-c}
give
\[
P_k
=
A\left(
-1+\frac{z_k^2}{2},
-\frac{z_k}{2},
-\frac{z_k}{2}
\right)
\longrightarrow
A(-1,0,0).
\]
Thus this invariant equality locus converges to a mixed fixed point.

Finally suppose \(q_{1,1}>q_{2,1}\). Then \(u_k\le u_1<1\), so
\(u_\infty<1\). If \(v_\infty=0\),
\eqref{eq:positive-kernel-region} would imply \(u_\infty\ge1\), a
contradiction. Thus \(v_\infty>0\). Suppose that
\(u_\infty+v_\infty>1\). Then
\(q_\infty=(1,u_\infty,v_\infty)^T\) has nonzero coordinates and satisfies
all three triangle inequalities strictly. The reconstruction formulas
therefore give a rank-two correlation matrix \(P_\infty\), and
\(P_k\to P_\infty\). Since \(u_k,v_k\) converge to positive limits,
\[
\frac{d_{2,k}}{d_{1,k}}
=
\frac{u_{k+1}}{u_k}
\longrightarrow1,
\qquad
\frac{d_{3,k}}{d_{1,k}}
=
\frac{v_{k+1}}{v_k}
\longrightarrow1.
\]
Continuity of the normalization factors gives
\(d_1(P_\infty)=d_2(P_\infty)=d_3(P_\infty)\). Applying
\eqref{eq:kernel-order-identity} to \((P_\infty,q_\infty)\) forces
\(q_{1,\infty}=q_{2,\infty}=q_{3,\infty}\), contradicting
\(u_\infty<1\).
Hence
\[
u_\infty+v_\infty=1.
\]
Substitution of \(q=(1,u_\infty,1-u_\infty)^T\) into
Lemma~\ref{lem:kernel-reconstruction} gives
\[
P_\infty=A(-1,-1,1).
\]
Therefore the matrix orbit converges to a patterned fixed point, up to the
initial permutation of the indices.

\medskip
\noindent
\textbf{Case 2: exactly one kernel coordinate is negative.}

By permutation equivariance, relabel the indices once so that
\[
q_{1,1}<0<q_{3,1}\le q_{2,1}.
\]
The order is preserved by Proposition~\ref{prop:projective-kernel-dynamics}.
Define the projective ratios
\[
x_k=-\frac{q_{1,k}}{q_{3,k}}>0,
\qquad
r_k=\frac{q_{2,k}}{q_{3,k}}\ge1.
\]
Thus \(\kappa(P_k)\) is represented by \((-x_k,r_k,1)^T\). The strict
triangle inequality gives
\begin{equation}\label{eq:one-negative-region}
r_k<1+x_k.
\end{equation}
Since
\[
-x_k<1\le r_k,
\]
Proposition~\ref{prop:projective-kernel-dynamics} gives
\[
d_{1,k}<d_{3,k}\le d_{2,k}.
\]
Thus
\[
x_{k+1}
=
x_k\frac{d_{1,k}}{d_{3,k}}
<x_k,
\qquad
r_{k+1}
=
r_k\frac{d_{2,k}}{d_{3,k}}
\ge r_k.
\]
Hence
\[
x_k\downarrow x_\infty,
\qquad
r_k\uparrow r_\infty,
\qquad
r_\infty\le1+x_\infty.
\]

If \(r_1=1\), equality of the two positive kernel coordinates is preserved.
The strict triangle inequality gives \(x_k<2\) for every \(k\), and since
\(x_k\) is decreasing,
\[
0\le x_\infty\le x_1<2.
\]
If \(x_\infty>0\), then \(q_\infty=(-x_\infty,1,1)^T\) has three
nonzero coordinates satisfying all triangle inequalities strictly. The
reconstruction formulas therefore give \(P_k\to P_\infty\), where
\(P_\infty\) is a rank-two correlation matrix with kernel spanned by
\(q_\infty\). Also
\(x_{k+1}/x_k\to1\), so \(d_{1,k}/d_{3,k}\to1\). By continuity,
\(d_1(P_\infty)=d_3(P_\infty)\), and
\eqref{eq:kernel-order-identity} at \((P_\infty,q_\infty)\) would give
\(-x_\infty=1\), a contradiction. Hence \(x_\infty=0\). Lemma~\ref{lem:kernel-reconstruction} gives
\[
P_k
=
A\left(
\frac{x_k}{2},
\frac{x_k}{2},
-1+\frac{x_k^2}{2}
\right)
\longrightarrow
A(0,0,-1),
\]
which is a mixed fixed point.

If \(r_1>1\), then \(r_\infty\ge r_1>1\). Therefore
\eqref{eq:one-negative-region} implies \(x_\infty>0\). The second
nontrivial triangle inequality is
\[
x_k<r_k+1.
\]
Because \(x_k\) is decreasing and \(r_k\) is increasing,
\[
r_k+1-x_k\ge r_1+1-x_1>0,
\]
and hence
\[
x_\infty<r_\infty+1.
\]
Thus the projective limit cannot approach the zero-sum boundary
\(x=r+1\), which would correspond, through the reconstruction formulas, to
the excluded all-ones vertex. If \(r_\infty<1+x_\infty\), then, together with
\(x_\infty<r_\infty+1\), all three triangle inequalities are strict for
\(q_\infty=(-x_\infty,r_\infty,1)^T\). The reconstruction formulas give a
rank-two correlation matrix \(P_\infty\) with \(P_k\to P_\infty\). The
convergence of \(x_k\) and \(r_k\) gives
\[
\frac{d_{1,k}}{d_{3,k}}
=
\frac{x_{k+1}}{x_k}
\longrightarrow1,
\qquad
\frac{d_{2,k}}{d_{3,k}}
=
\frac{r_{k+1}}{r_k}
\longrightarrow1.
\]
By continuity,
\(d_1(P_\infty)=d_2(P_\infty)=d_3(P_\infty)\). Applying
\eqref{eq:kernel-order-identity} to \((P_\infty,q_\infty)\) would force
all three kernel coordinates to be equal, impossible because one coordinate
is negative. Consequently,
\[
r_\infty=1+x_\infty.
\]
Substitution of \(q=(-x_\infty,1+x_\infty,1)^T\) into the reconstruction formulas yields
\[
P_\infty=A(1,-1,-1).
\]
Thus the matrix orbit converges to a patterned fixed point, up to permutation.

\medskip
\noindent
\textbf{Case 3: one kernel coordinate is zero.}

By Lemma~\ref{lem:kernel-triangle-constraints}, two kernel coordinates
cannot vanish. Assume after permutation that the zero coordinate is the third
one. Because the diagonal update in
\eqref{eq:successive-kernel-representatives} is positive, this coordinate
remains zero for every subsequent rank-two iterate. For a representative
\(q=(q_1,q_2,0)^T\), Lemma~\ref{lem:kernel-triangle-constraints} gives
\[
|q_1|=|q_2|.
\]
The alternative \(q_2=-q_1\) would give
\[
s(q)=q_1+q_2+q_3=0,
\]
so Proposition~\ref{prop:projective-kernel-dynamics} would force the next
iterate to have rank one. This is excluded by the permanent-rank-two
assumption \eqref{eq:permanent-rank-two}; in the complementary branch of the
proof, such an orbit has already been settled because the resulting rank-one
state is a patterned fixed point.

It therefore remains to consider \(q_2=q_1\). Normalize the projective representative to
\(q=(1,1,0)^T\). For the corresponding Gram factor
\(P=XX^T\) with unit rows \(x_1^T,x_2^T,x_3^T\), the kernel relation gives
\(x_2=-x_1\). Hence the correlation matrix lies on the invariant boundary
curve
\[
P=A(-1,-t,t),
\qquad
|t|<1.
\]
Lemma~\ref{lem:mixed-invariant-curve} gives
\[
t^+=g(t)
=
\frac{4t}{\sqrt{3t^4+10t^2+3}}.
\]
For \(0<|t|<1\),
\[
\frac{g(t)^2}{t^2}
=
\frac{16}{3t^4+10t^2+3}
>1
\]
and
\[
1-g(t)^2
=
\frac{3(t^2-1)^2}{3t^4+10t^2+3}
>0.
\]
Hence
\[
|t|<|g(t)|<1.
\]
The sign of \(t\) is preserved, so \(|t_k|\) increases to a limit. By
continuity, the only positive limiting modulus is \(1\). Thus
\(t_k\to1\) or \(t_k\to-1\), and the orbit converges to a patterned
rank-one fixed point. If \(t=0\), the starting state is the mixed fixed point
\((-1,0,0)\).

The three cases exhaust all rank-two kernel configurations. Therefore every
admissible orbit converges to one of the seven fixed points.
\end{proof}

\begin{corollary}[Coordinate form of global convergence]
\label{cor:coordinate-global-convergence}
For every
\[
p_0\in\mathcal E_3^{\mathrm{nd}},
\]
the orbit \(T^k(p_0)\) converges to one of the seven points in
\eqref{eq:complete-fixed-point-set}.
\end{corollary}

\begin{proof}
By the definition of the induced coordinate map,
\[
A(T(p))=C(A(p))
\]
for every \(p\in\mathcal E_3^{\mathrm{nd}}\). Hence, by induction,
\[
A(T^k(p_0))=C^k(A(p_0))
\qquad(k\ge0).
\]
The parametrization \(p\mapsto A(p)\) is an affine homeomorphism from
\(\mathbb R^3\) onto the affine space of symmetric \(3\times3\) matrices
with unit diagonal, and it restricts to a homeomorphism from
\(\mathcal E_3\) onto \(\mathcal C_3\). The conclusion therefore follows
directly from Theorem~\ref{thm:global-convergence}.
\end{proof}

\begin{corollary}[Absence of nontrivial recurrent limit sets]
\label{cor:no-nontrivial-periodic-orbits}
The \(3\times3\) Pearson correlation map has no periodic orbit of prime period
greater than one. More generally, for every
\(p\in\mathcal E_3^{\mathrm{nd}}\), the omega-limit set
\(\omega(p)\) from Definition~\ref{def:omega-limit-set} is a singleton
fixed point.
\end{corollary}

\begin{proof}
This is immediate from
Corollary~\ref{cor:coordinate-global-convergence}.
\end{proof}

For the measure-theoretic statements below, we return to the coordinate
representation
\[
p=(a,b,c)\in\mathcal E_3,
\]
because \(\lambda_3\) denotes three-dimensional Lebesgue measure in these
coordinates. Here \(\operatorname{int}\mathcal E_3\) and
\(\partial\mathcal E_3\) denote the Euclidean interior and boundary,
respectively, in \(\mathbb R^3\).

\begin{lemma}[Necessary symmetry for a non-patterned limit]
\label{lem:nonpatterned-symmetry}
Let
\[
p\in\operatorname{int}\mathcal E_3.
\]
If the orbit \(T^k(p)\) converges to one of the four non-patterned fixed
points, then
\[
T(p)\in\Sigma,
\]
where
\[
\Sigma
=
\left\{
(a,b,c)\in\mathcal E_3:
a=b\text{ or }a=c\text{ or }b=c
\right\}.
\]
\end{lemma}

\begin{proof}
Set
\[
P_0=A(p),
\qquad
P_1=C(P_0)=A(T(p)).
\]
Since \(p\in\operatorname{int}\mathcal E_3\), the matrix \(P_0\) is positive
definite and therefore has rank three. Proposition~\ref{prop:rank-bound}
gives
\[
\operatorname{rank}P_1\le2.
\]
If \(\operatorname{rank}P_1=1\), then
Corollary~\ref{cor:rank-one-correlation-states} implies that \(P_1\) is a
nondegenerate sign-valued correlation matrix, hence one of the three
patterned fixed points. In that case the orbit cannot converge to a
non-patterned fixed point. Therefore a non-patterned limit requires
\[
\operatorname{rank}P_1=2.
\]

Let \(q=(q_1,q_2,q_3)^T\neq0\) span \(\ker P_1\). The rank-two case analysis
in the proof of Theorem~\ref{thm:global-convergence} shows that convergence
to a mixed fixed point can occur only on one of the invariant equality
loci
\[
q_1=q_2,\qquad q_1=q_3,\qquad q_2=q_3,
\]
including the corresponding zero-coordinate endpoint cases. When all three
kernel coordinates are nonzero, Lemma~\ref{lem:kernel-reconstruction}
gives
\[
q_1=q_2\Longrightarrow b_1=c_1,\qquad
q_1=q_3\Longrightarrow a_1=c_1,\qquad
q_2=q_3\Longrightarrow a_1=b_1,
\]
where
\[
T(p)=(a_1,b_1,c_1).
\]
The zero-coordinate mixed fixed points themselves also have two equal
off-diagonal coordinates. The same rank-two analysis shows that convergence
to the equicorrelation point is possible only when \(P_1\) is already the
equicorrelation fixed point, whose three off-diagonal coordinates are equal.
Consequently,
\[
T(p)\in\Sigma.
\]
\end{proof}

\begin{corollary}[Full-measure convergence to the patterned orbit]
\label{cor:full-measure-patterned-convergence}
For \(\lambda_3\)-almost every
\[
p=(a,b,c)\in\mathcal E_3^{\mathrm{nd}},
\]
the orbit \(T^k(p)\) converges to one of the three patterned fixed points.
\end{corollary}

\begin{proof}
Let
\[
E_{\mathrm{np}}
=
\left\{
p\in\mathcal E_3^{\mathrm{nd}}:
T^k(p)\text{ converges to a non-patterned fixed point}
\right\}.
\]
By Lemma~\ref{lem:nonpatterned-symmetry},
\[
E_{\mathrm{np}}\cap\operatorname{int}\mathcal E_3
\subseteq
T^{-1}(\Sigma).
\]

On the open convex set \(\operatorname{int}\mathcal E_3\), the coordinate
functions \(T_1,T_2,T_3\) are real analytic. Define
\[
\begin{aligned}
F(a,b,c)
={}&
\bigl(T_1(a,b,c)-T_2(a,b,c)\bigr)\\
&\times
\bigl(T_1(a,b,c)-T_3(a,b,c)\bigr)\\
&\times
\bigl(T_2(a,b,c)-T_3(a,b,c)\bigr).
\end{aligned}
\]
By the definition of \(\Sigma\),
\[
T^{-1}(\Sigma)\cap\operatorname{int}\mathcal E_3
=
\left\{
p\in\operatorname{int}\mathcal E_3:F(p)=0
\right\}.
\]

The function \(F\) is not identically zero. At
\[
(a,b,c)=\left(\frac12,\frac14,0\right),
\]
which lies in \(\operatorname{int}\mathcal E_3\), direct substitution gives
\[
T_1=\frac{\sqrt{21}}{14},
\qquad
T_2=-\frac{11\sqrt{91}}{182},
\qquad
T_3=-\frac{2\sqrt{39}}{13},
\]
and these values are pairwise distinct. Since
\(\operatorname{int}\mathcal E_3\) is connected, the zero set of the
nonzero real-analytic function \(F\) has three-dimensional Lebesgue measure
zero; see, for example, \cite{mityagin2020}. Hence
\[
\lambda_3\!\left(
E_{\mathrm{np}}\cap\operatorname{int}\mathcal E_3
\right)=0.
\]

The boundary \(\partial\mathcal E_3\) is contained in the zero set of the
nonzero polynomial
\[
1+2abc-a^2-b^2-c^2,
\]
so
\[
\lambda_3(\partial\mathcal E_3)=0.
\]
Therefore
\[
\lambda_3(E_{\mathrm{np}})=0.
\]
By Corollary~\ref{cor:coordinate-global-convergence}, every admissible
coordinate orbit converges to one of the seven fixed points. Thus, outside
the null set \(E_{\mathrm{np}}\), the limit must be one of the three
patterned fixed points.
\end{proof}
\begin{remark}[Transient step-size growth and global convergence]
\label{rem:transient-step-growth}
Global convergence does not imply monotone decay of the successive
step sizes
\[
\Delta_k
=
\|P_{k+1}-P_k\|_F.
\]
The invariant mixed boundary curve from
Lemma~\ref{lem:mixed-invariant-curve} gives an exact example.

Let
\[
P_k=A(-1,-t_k,t_k),
\qquad
t_{k+1}=g(t_k),
\]
where
\[
g(t)
=
\frac{4t}{\sqrt{3t^4+10t^2+3}}.
\]
Since only the two off-diagonal coordinates \(-t_k\) and \(t_k\)
vary along this curve, with each occurring twice in the symmetric
matrix,
\[
\Delta_k
=
2|t_{k+1}-t_k|.
\]
Consequently,
\[
\frac{\Delta_{k+1}}{\Delta_k}
=
\frac{|g(g(t_k))-g(t_k)|}
{|g(t_k)-t_k|}
\]
whenever
\[
0<|t_k|<1.
\]

At the mixed fixed point,
\[
g(0)=0,
\qquad
g'(0)=\frac{4}{\sqrt3}>1.
\]
Set
\[
\lambda=\frac{4}{\sqrt3}.
\]
Then
\[
g(t)=\lambda t+o(t)
\qquad
\text{as }t\to0,
\]
and hence
\[
g(g(t))-g(t)
=
\lambda(\lambda-1)t+o(t),
\]
whereas
\[
g(t)-t
=
(\lambda-1)t+o(t).
\]
Therefore
\begin{equation}
\label{eq:mixed-step-growth-limit}
\lim_{t\to0}
\frac{|g(g(t))-g(t)|}
{|g(t)-t|}
=
\frac{4}{\sqrt3}>1.
\end{equation}
Thus sufficiently small nonzero perturbations of the mixed fixed point
exhibit transient growth of successive Frobenius step sizes.

This behavior is fully compatible with
Theorem~\ref{thm:global-convergence}. Along the same invariant curve,
every trajectory with
\[
0<|t_0|<1
\]
ultimately converges to one of the patterned endpoints. Moreover,
\[
g'(1)=g'(-1)=0,
\]
so the curve exhibits strong contraction near those patterned fixed
points.

Thus the three-dimensional theory provides an exact mechanism by which
transient step-size overshoots can occur within a globally convergent
orbit: local expansion near an unstable fixed point may be followed by
strong contraction toward a stable patterned fixed point. This
observation neither quantifies the prevalence of such overshoots under
a sampling law nor asserts that the same mechanism explains all
overshoots observed numerically in higher dimensions.
\end{remark}

\section{Basin Geometry and Computational Illustration}\label{sec:basins}

For a patterned fixed point \(p_\ast\), define
\[
\mathcal B(p_\ast)
=
\left\{
p\in\mathcal E_3^{\mathrm{nd}}:
T^k(p)\to p_\ast
\right\}.
\]
For \(E\subseteq\mathcal E_3\) and \(\sigma\in S_3\), write
\[
\sigma\cdot E
=
\{\sigma\cdot p:p\in E\}
\]
for the induced action on subsets. By forward invariance, no separate
condition ensuring existence of future iterates is needed.

\begin{proposition}[Openness, symmetry, and exact basin measures]
\label{prop:basin-symmetry}
For every patterned fixed point \(p_\ast\), the basin
\(\mathcal B(p_\ast)\) is relatively open in
\(\mathcal E_3^{\mathrm{nd}}\) and is therefore Lebesgue measurable.
For every \(\sigma\in S_3\),
\begin{equation}\label{eq:basin-equivariance}
\sigma\cdot\mathcal B(p_\ast)=\mathcal B(\sigma\cdot p_\ast).
\end{equation}
The three patterned basins have equal Lebesgue measure. Moreover,
\begin{equation}\label{eq:exact-one-third-basin-measure}
\lambda_3(\mathcal B(p_\ast))
=
\frac13\lambda_3(\mathcal E_3)
\end{equation}
for each patterned fixed point \(p_\ast\).
\end{proposition}

\begin{proof}
Let \(p_\ast\) be a patterned fixed point. By Theorem~\ref{thm:stability},
there is a relative neighborhood \(U\) of \(p_\ast\) such that every orbit
starting in \(U\) converges to \(p_\ast\). If
\(p\in\mathcal B(p_\ast)\), then
\[
T^N(p)\in U
\]
for some \(N\). Since \(T\) is a continuous self-map of
\(\mathcal E_3^{\mathrm{nd}}\), the set
\[
(T^N)^{-1}(U)
\]
is a relatively open neighborhood of \(p\) contained in
\(\mathcal B(p_\ast)\). Thus \(\mathcal B(p_\ast)\) is relatively open in
\(\mathcal E_3^{\mathrm{nd}}\). Since
\[
\mathcal E_3^{\mathrm{nd}}
=
\mathcal E_3\setminus\{(1,1,1)\}
\]
is a Borel subset of \(\mathbb R^3\), every relatively open subset of
\(\mathcal E_3^{\mathrm{nd}}\) is Borel and hence Lebesgue measurable.

Equivariance gives
\[
T^k(\sigma\cdot p)=\sigma\cdot T^k(p).
\]
Therefore
\[
T^k(p)\to p_\ast
\quad\Longleftrightarrow\quad
T^k(\sigma\cdot p)\to\sigma\cdot p_\ast,
\]
which proves equation~\eqref{eq:basin-equivariance}. Coordinate
permutations preserve \(\lambda_3\) and act transitively on the three
patterned fixed points, so the three patterned basins have equal measure.

The three basins are pairwise disjoint, and
Corollary~\ref{cor:full-measure-patterned-convergence} states that their
union has full measure in \(\mathcal E_3^{\mathrm{nd}}\). Since the excluded
point \((1,1,1)\) is a singleton,
\[
\lambda_3(\mathcal E_3^{\mathrm{nd}})
=
\lambda_3(\mathcal E_3).
\]
Thus, if \(m\) denotes the common measure of the three patterned basins,
\[
3m=\lambda_3(\mathcal E_3),
\]
and equation~\eqref{eq:exact-one-third-basin-measure} follows.
\end{proof}

\subsection{Reproducible Computational Protocol}

The basin experiment was rerun with a deterministic implementation under
Python 3.12.0 using NumPy 1.26.4, SciPy 1.17.1, and Matplotlib 3.8.0.
Using NumPy's \texttt{default\_rng} generator with seed \(20260714\),
the code generated \(80{,}000\) independent triples uniformly in
\([-1,1]^3\).

A triple was retained if it satisfied
\[
1+2abc-a^2-b^2-c^2
\ge
-10^{-10}.
\]
The small negative allowance is solely a floating-point tolerance for
the exact elliptope condition in
equation~\eqref{eq:elliptope-e3}. For the reported seeded sample, no retained
point relied on this tolerance: all \(49{,}164\) retained triples satisfied
the determinant condition strictly, and the minimum determinant among them
was
\[
3.7090516\times10^{-5}>0.
\]

For every admissible initial matrix \(A_0\), the iteration
\[
A_{k+1}
=
C(A_k)
\]
was continued until
\[
\|A_{k+1}-A_k\|_{\max}
<
10^{-10},
\]
or until \(1000\) iterations.

No artificial lower bound was imposed on the centered-row norms. The exact forward invariance theorem
proves that an admissible trajectory cannot become
Pearson-degenerate at a finite iterate. The implementation nevertheless
retained a threshold of \(10^{-14}\) as a finite-precision diagnostic; a
centered-row norm below that value would have been recorded as an
undefined-map event. After the stopping
criterion was met, the final triple was compared with all seven fixed
points using absolute and relative tolerances \(10^{-6}\). The
classification procedure did not use the stability labels.

The cube filter retained \(49{,}164\) admissible initial conditions and
rejected \(30{,}836\) points. Every admissible trajectory matched one of
the seven listed fixed points. There were no unmatched,
maximum-iteration, or undefined-map cases. The largest iteration count
was \(19\), and the smallest centered-row norm encountered over all
computed steps was approximately
\[
1.45\times10^{-2}.
\]

A stricter validation run used the identical \(80{,}000\) sampled
triples and changed only:

\begin{itemize}
\item the stopping tolerance to \(10^{-12}\);
\item the matching tolerance to \(10^{-7}\);
\item the iteration limit to \(5000\).
\end{itemize}

A trajectory-by-trajectory comparison produced zero classification
mismatches across all \(49{,}164\) admissible initial conditions. The
validation run again produced no unmatched, maximum-iteration, or
undefined-map cases.

\subsection{Results}

Table~\ref{tab:basins} reports the observed basin frequencies for the
seeded experiment. Only the three patterned fixed points occurred as
numerical limits; none of the sampled trajectories was classified as
converging to a mixed or equicorrelation fixed point.

\begin{table}[htbp]
\centering
\caption{Observed basin frequencies for the \(49{,}164\) admissible
initial conditions in the seeded experiment. Each patterned proportion
has a 95\% binomial normal-approximation confidence-interval half-width
of approximately \(0.42\) percentage points.}
\label{tab:basins}
\begin{tabular}{cllrr}
\toprule
Label & Family & Coordinates \((a,b,c)\) & Count & Frequency (\%)\\
\midrule
FP1 & Patterned
& \((1,-1,-1)\)
& \(16{,}342\) & \(33.24\)\\
FP2 & Mixed
& \((0,0,-1)\)
& \(0\) & \(0.00\)\\
FP3 & Patterned
& \((-1,1,-1)\)
& \(16{,}428\) & \(33.41\)\\
FP4 & Mixed
& \((0,-1,0)\)
& \(0\) & \(0.00\)\\
FP5 & Equicorrelation
& \(\left(-\frac12,-\frac12,-\frac12\right)\)
& \(0\) & \(0.00\)\\
FP6 & Patterned
& \((-1,-1,1)\)
& \(16{,}394\) & \(33.35\)\\
FP7 & Mixed
& \((-1,0,0)\)
& \(0\) & \(0.00\)\\
\bottomrule
\end{tabular}
\end{table}

For each patterned basin, let \(\widehat p\) denote its observed sample
proportion. The corresponding approximate \(95\%\) confidence interval is
computed as
\[
\widehat p
\pm
1.96
\sqrt{
\frac{
\widehat p(1-\widehat p)
}{N}
},
\qquad
N=49{,}164.
\]

No sampled trajectory was numerically classified as converging to a mixed
point or to the equicorrelation point. This is consistent with
Theorem~\ref{thm:global-convergence} and
Corollary~\ref{cor:full-measure-patterned-convergence}: the four
non-patterned fixed points have only a Lebesgue-null set of initial conditions
converging to them. Proposition~\ref{prop:basin-symmetry} proves that the three
patterned basins have exactly equal measure. The small differences among the
observed frequencies are therefore finite-sample fluctuations around the exact
one-third basin proportions.
Figure~\ref{fig:barchart} visualizes these frequencies, both for the
seven individual fixed points and for the comparison between the three
patterned basins and the four unstable fixed points.

\begin{figure}[!htbp]
\centering
\includegraphics[
width=\textwidth,
height=0.9\textheight,
keepaspectratio
]{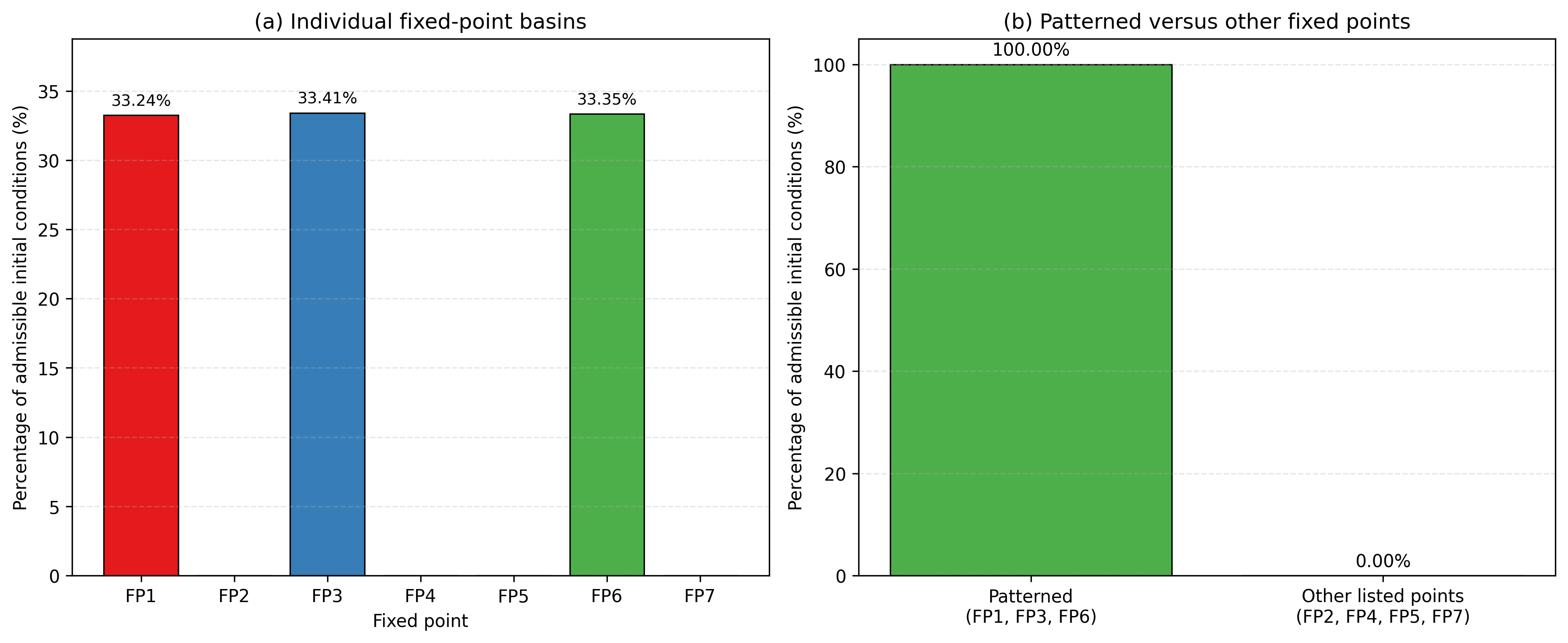}
\caption{Observed basin frequencies in the seeded experiment.
Panel~(a) reports the percentage associated with each of the seven
fixed points. Panel~(b) compares the combined patterned basins with the
four other fixed points. Every one of the \(49{,}164\) admissible
sampled trajectories was matched to FP1, FP3, or FP6.}
\label{fig:barchart}
\end{figure}

Figure~\ref{fig:slices} shows two-dimensional views of the basins in
thin bands centered at selected values of the initial coordinate
\(c_0\).

\begin{figure}[!htbp]
\centering
\includegraphics[
width=\textwidth,
height=0.70\textheight,
keepaspectratio
]{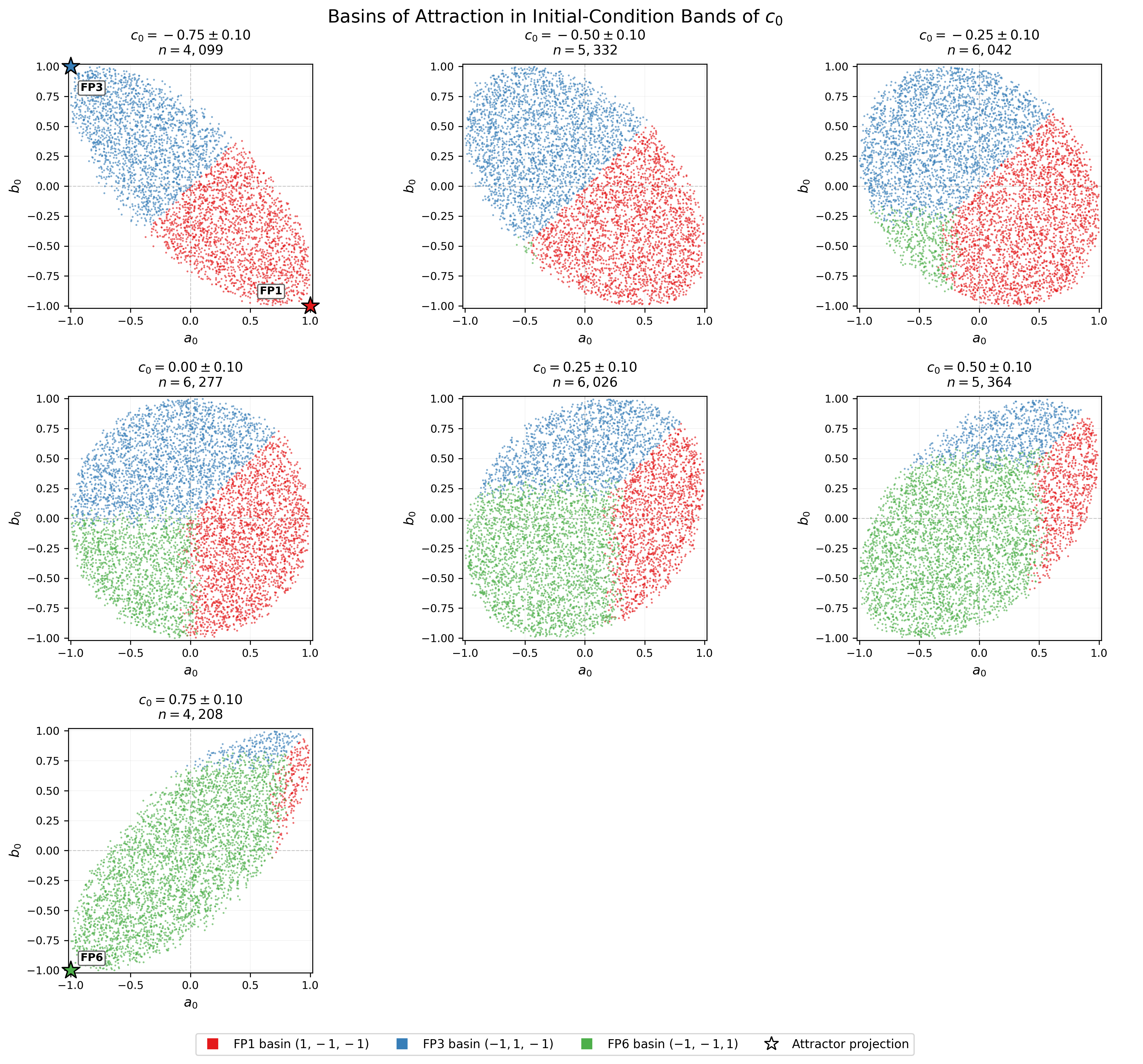}
\caption{Two-dimensional views of the observed basins. A panel centered
at \(c_*\) contains admissible initial conditions satisfying
\(\lvert c_0-c_*\rvert<0.10\), so the panels represent thin bands
rather than exact planar slices. The displayed \(n\) is the number of
initial conditions retained in that band. Colors identify the numerical
limit: red for FP1, blue for FP3, and green for FP6. Star markers show
the projections of the patterned fixed points in the displayed bands
nearest their actual \(c\)-coordinates, namely \(c=-1\) for FP1 and
FP3 and \(c=1\) for FP6. The final panel contains the legend.}
\label{fig:slices}
\end{figure}

The slice views in Figure~\ref{fig:slices} show how the three observed basins
are distributed across different bands of the initial coordinate \(c_0\).
Although the visible basin geometry changes from one band to another, only
FP1, FP3, and FP6 occur as numerical limits. The computation is therefore
consistent with the exact global theory: every admissible orbit converges to a
fixed point, while almost every initial condition converges to the patterned
orbit. The numerical experiment is not used in the proof of global
convergence; it provides an independent reproducibility check and a geometric
visualization of the three patterned basins comprising the full-measure union.

\section{Discussion}\label{sec:discussion}

The dimension-three dynamics is now completely determined at the level of
forward existence and asymptotic convergence. Theorem~\ref{thm:forward-invariance}
excludes finite-time Pearson degeneracy on the natural state space, while
Theorem~\ref{thm:global-convergence} proves that every admissible
\(3\times3\) trajectory converges to a fixed point. Combined with
Theorem~\ref{thm:complete-classification}, this gives a complete list of all
possible limits: three rank-one patterned points, three rank-two mixed points,
and one rank-two equicorrelation point.

The key global mechanism is distinct from the fixed-point classification.
The classification assumes stationarity and converts equality of two Gram
factorizations into an eigenline problem for a two-dimensional linear
operator. The convergence proof instead follows the one-dimensional kernel of
a general rank-two iterate. The exact projective update
\[
[q^+]=[D(P)q],
\]
together with
\eqref{eq:kernel-order-identity}, converts the boundary dynamics into
projectively monotone evolution. The resulting order structure rules out
nontrivial periodic or recurrent limit sets and forces each trajectory toward
a fixed configuration.

The local stability classification remains essential for understanding which
limits are dynamically typical. The three patterned points are locally
asymptotically stable. The mixed and equicorrelation points are Lyapunov
unstable relative to the natural domain. The global kernel analysis shows
more: convergence to these four unstable points is confined to a
Lebesgue-null exceptional set. Corollary~\ref{cor:full-measure-patterned-convergence}
therefore establishes full-measure convergence to the patterned orbit as a theorem.

Permutation equivariance then gives an exact measure-theoretic statement.
The three patterned basins are relatively open, measurable, and congruent
under coordinate permutations. Their union has full Lebesgue measure, so each
basin occupies exactly one third of the elliptope in the three-dimensional
Lebesgue measure. The numerical frequencies in Table~\ref{tab:basins} are
therefore interpreted as sampling fluctuations around an exact symmetry law,
not as estimates of an unknown basin ratio.

The general-dimensional structural results remain separate from the
three-dimensional convergence theorem. For every \(n\ge2\), the exact rank
identity
\[
\operatorname{rank}C(A)=\operatorname{rank}(AH_n)
\]
places Chen's rank reduction in an explicit row-wise Gram framework. The
all-ones matrix is the unique Pearson-degenerate point of the elliptope, and
the nondegenerate elliptope is forward invariant. Proposition~\ref{prop:sign-fixed-points} also gives the complete nondegenerate
sign-valued fixed-point family, of cardinality \(2^{n-1}-1\).

What is special in dimension three is that one Pearson step forces rank at
most two and hence leaves a one-dimensional kernel whose projective direction
parametrizes the generic boundary state. For \(n\ge4\), the rank after one
step can be as large as \(n-1\), the kernel no longer parametrizes the state,
and the two-dimensional Cayley--Hamilton identity used in
Proposition~\ref{prop:projective-kernel-dynamics} has no direct analogue with the same
ordering consequence. Thus the present proof does not establish global
convergence in higher dimensions.

The reproducible computation remains useful even after the analytical
convergence theorem. It verifies the implementation against the exact theory,
illustrates the geometry of the three patterned basins, and records the short
transient lengths observed under uniform sampling. It is no longer required
to support a conjectural global statement.

\section{Open Problems}
\label{sec:open-problems}

The global convergence problem is resolved here for \(n=3\), but several
questions remain open in higher dimensions and in the finer geometry of the
exceptional sets.

\begin{problem}[Higher-dimensional fixed-point classification]
For each \(n\ge4\), determine the complete fixed-point set
\[
\operatorname{Fix}(C)
=
\{A\in\mathcal C_n^{\mathrm{nd}}:C(A)=A\}.
\]
Proposition~\ref{prop:rank-bound} shows that every fixed point satisfies
\(\operatorname{rank}A\le n-1\), and
Proposition~\ref{prop:sign-fixed-points} classifies all nondegenerate
sign-valued fixed points. It remains to determine whether additional
higher-rank fixed points exist and, if so, to classify them.
\end{problem}

\begin{problem}[Global convergence in higher dimensions]
For \(n\ge4\), does every orbit in \(\mathcal C_n^{\mathrm{nd}}\) converge?
If not, characterize the possible nontrivial \(\omega\)-limit sets and the
initial conditions that generate them.
\end{problem}

\begin{problem}[Stability in higher dimensions]
Which higher-dimensional fixed-point families are relatively Lyapunov stable
or locally asymptotically stable? In particular, can a fixed point of rank
greater than one be locally asymptotically stable for some \(n\ge4\)?
\end{problem}

\begin{problem}[Exact exceptional sets in dimension three]
Determine the stable sets of the three mixed fixed points and of the
equicorrelation fixed point explicitly, and describe how these exceptional
stable sets sit inside the boundaries of the three open patterned basins.
\end{problem}

\begin{problem}[Basin boundaries]
Give an explicit geometric or algebraic description of the boundaries of the
three patterned basins in \(\mathcal E_3\). Determine which stable manifolds or
invariant curves of the unstable fixed points form part of these boundaries.
\end{problem}

\begin{problem}[Asymptotic degeneracy in higher dimensions]
Theorem~\ref{thm:forward-invariance} excludes finite-time arrival at the
all-ones matrix in every dimension. Can a higher-dimensional forward orbit
approach \(\mathbf1\mathbf1^T\) asymptotically along a subsequence, or can
this also be excluded analytically?
\end{problem}

\section*{Declarations}

\paragraph{Ethics approval and consent to participate.}
Not applicable.

\paragraph{Consent for publication.}
Not applicable.

\paragraph{Availability of data and materials.}
The complete reproducibility package, including the source code, numerical
summaries, raw trajectory classifications, validation outputs, and
figure-generation scripts used in this work, is archived on Zenodo as
Version~2.0.0 \cite{alhajjhassan2026reproducibility}. The Zenodo record is
publicly accessible, while the archived files are restricted during manuscript
review and will be made publicly available upon publication of the associated
manuscript. Access to the restricted files may be provided to editors and
reviewers upon request. The source code is also maintained in the private
GitHub repository at
\url{https://github.com/IshrakAlhajjHassan/fixedpoints}
and will be made publicly available upon publication.

\paragraph{Competing interests.}
The author declares no competing interests.

\paragraph{Funding.}
No external funding was received for this study.

\paragraph{Author contributions.}
The sole author was responsible for the conception, mathematical
analysis, computational implementation, validation, visualization, and
preparation of the manuscript.

\paragraph{Acknowledgments.}
The author gratefully acknowledges Prof.~Pasha Zusmanovich for his
valuable guidance, careful comments, and constructive discussions
throughout the development of this work.

\end{document}